\documentclass[reqno]{amsart}

\usepackage[applemac]{inputenc}

\usepackage{amssymb}
\usepackage{graphicx}
\usepackage[cmtip,all]{xy}
\usepackage{verbatim}
\usepackage[toc,page]{appendix}
\usepackage{geometry}
\usepackage{graphicx}
\usepackage{mathrsfs}
\usepackage{pinlabel}
\usepackage{color}
\usepackage{mathtools}
\usepackage{stmaryrd}
\usepackage{mathabx}

\DeclareMathAlphabet{\mathpzc}{OT1}{pzc}{m}{it}

\newtheorem{theorem}{Theorem}[section]
\newtheorem{lemma}[theorem]{Lemma}
\newtheorem{prop}[theorem]{Proposition}

\theoremstyle{definition}
\newtheorem{definition}[theorem]{Definition}
\newtheorem{example}[theorem]{Example}
\newtheorem{problem}[theorem]{Problem} 
\newtheorem{question}[theorem]{Question} 
\newtheorem{corollary}[theorem]{Corollary}

\theoremstyle{remark}
\newtheorem{remark}[theorem]{Remark}

\numberwithin{equation}{section}

\newcommand{\set}[1]{\ensuremath{\mathbb{#1}}}

\newcommand{\slantone}[2]{{\raisebox{.1em}{$#1$}\left/\raisebox{-.1em}{$#2$}\right.}}
\newcommand*{\defeq}{\mathrel{\vcenter{\baselineskip0.5ex \lineskiplimit0pt
                     \hbox{\scriptsize.}\hbox{\scriptsize.}}}%
                     =}

\newcommand\asim{\mathrel{%
  \ooalign{\raise0.1ex\hbox{$\sim$}\cr\hidewidth\raise-0.8ex\hbox{\scalebox{0.9}{$\scriptstyle{x}$}}\hidewidth\cr}}}

\newcommand{\proj}[1]{\ensuremath{\mathbb{P}^{#1}}}

\newcommand{\mani}{\ensuremath{\mathpzc{M}}}
\newcommand{\manir}{\ensuremath{\mathpzc{M}_{red}}}
\newcommand{\stsheaf}{\ensuremath{\mathcal{O}_{\mathpzc{M}}}}
\newcommand{\stsheafred}{\ensuremath{\mathcal{O}_{\mathpzc{M}_{red}}}}

\newcommand{\beq}{\begin{equation}}
\newcommand{\eeq}{\end{equation}}
\newcommand{\bear}{\begin{eqnarray}}
\newcommand{\eear}{\end{eqnarray}}

\begin{document}

\title{Non Projected Calabi-Yau Supermanifolds over $\mathbb{P}^{2}$}


%    author one information
\author{Sergio L. Cacciatori}
\address{Dipartimento di Scienza e Alta Tecnologia, Universit\`a dell'Insubria, Via Valleggio 11, 22100 Como, Italy and INFN, Sezione di Milano, Via Celoria 16, 20133 Milano, Italy}
\email{sergio.cacciatori@uninsubria.it}
%\thanks{}

%   author two information
\author{Simone Noja}
\address{Dipartimento di Scienza e Alta Tecnologia, Universit\`a dell'Insubria, Via Valleggio 11, 22100 Como, Italy and INFN, Sezione di Milano, Via Celoria 16, 20133 Milano, Italy}
\email{simone.noja@uninsubria.it}
%\thanks{}

% author three information
\author{Riccardo Re}
\address{Dipartimento di Scienza e Alta Tecnologia, Universit\`a dell'Insubria, Via Valleggio 11, 22100 Como, Italy }
\email{riccardo.re@uninsubria.it}

%\author{Sergio L. Cacciatori\inst{1,2}  \and Simone Noja\inst {3} \and Riccardo Re\inst{4}}
%\institute{Dipartimento di Scienza e Alta Tecnologia, Universit\`a dell'Insubria, Via Valleggio 11, I-22100 Como, Italy \and INFN, Sezione di Milano, Via Celoria 16, I-20133 Milano, Italy \and 
%Dipartimento di Matematica, Universit\`a degli Studi di Milano, Via Saldini 50, IT-20133 Milano, Italy \and  Dipartimento di Matematica e Informatica, Universit\`a degli Studi di Catania,
%Viale Andrea Doria 6, 95125 Catania, Italy}
%\title{Non Projected Calabi-Yau Supermanifolds over $\mathbb{P}^{2}$}
%\titlerunning{Non Projected Calabi-Yau Supermanifolds over $\mathbb{P}^{2}$}
%\authorrunning{Simone Noja et al.}

%\subjclass[2000]{Primary }
%    The 2010 edition of the Mathematics Subject Classification is
%    now available.  If you are citing a classification from the
%    new scheme, use the following input coding instead.
%\subjclass[2010]{Primary }

%\date{}

\begin{abstract} 
We start a systematic study of non-projected supermanifolds, concentrating on supermanifolds with fermionic dimension 2 and with the reduced manifold  a complex projective space. We show that all the non-projected supermanifolds of dimension $2|2$ over $\mathbb P^2$ are completely characterised by a non-zero cohomology class $\omega\in H^1(\mathcal{T}_{\mathbb{P}^2}(-3))$ and by a locally free sheaf $\mathcal F_\mani$ of rank $0|2$, satisfying $Sym^2 \mathcal{F}_\mani \cong K_{\proj 2}$. Denoting such supermanifolds with $\mathbb P^{2}_\omega(\mathcal F_\mani)$, we show that all of them are Calabi-Yau supermanifolds and, when $\omega \neq 0$, they are non-projective, that is they cannot be embedded into any projective 
superspace $\proj {n|m}$. Instead, we show that every non-projected supermanifold over $\proj 2$ admits an embedding into a super Grassmannian. By contrast, we give an example of a supermanifold $\mathbb P^{2}_\omega(\mathcal F_\mani)$ that cannot be embedded in any of the $\Pi$-projective superspaces $\mathbb P^{n}_{\Pi}$ introduced by Manin and 
Deligne. However, we also show  that when $\mathcal F_\mani$ is the cotangent bundle over $\proj 2$, then the non-projected $\mathbb{P}^2_\omega(\mathcal F_\mani)$ and the $\Pi$-projective plane $\mathbb P^{2}_{\Pi}$ do coincide.
\end{abstract}

\maketitle

\tableofcontents

\section{Introduction} 
In the present paper we take on a systematic study of supergeometry with particular attention to non-projected supermanifolds having odd dimension 2.  Once we have settled our conventions, we consider the 
obstruction of supermanifolds to be projected. We then concentrate on supermanifolds having a complex 
projective space $\mathbb{P}_{\mathbb{C}}^n$ as the associated reduced manifold, and we study the most general conditions for such supermanifolds to be non-projected. We get that only varieties of bosonic dimension 1 and 2 admit non-projected structures. We completely classify non-projected 
supermanifolds over $\mathbb P^1$ having odd dimension $2$, and, over $\mathbb P^2$, we provide a necessary and sufficient condition for the supermanifold to be non-projected. Remarkably, those supermanifolds are Calabi-Yau supermanifolds \cite{NCR}. Moreover, we prove that all these supermanifolds are non-projective, \emph{i.e.}\ they cannot be embedded into a split projective superspace.\\
Indeed we think that the role of embedding spaces in supergeometry should be played by super Grassmannians: along this line of thought, we prove that all of the non-projected supermanifolds over $\proj 2$ can be embedded into a certain super 
Grassmannian, Theorem \ref{embeddingth}. \\
We then study in detail the supermanifolds corresponding to two different choices of the fermionic sheaf $\mathcal F_\mani$ satisfying the aforementioned condition: the case when $\mathcal F_\mani$ is the decomposable sheaf 
$\mathcal O_{\proj 2}(-1)\oplus \mathcal O_{\proj 2}(-2)$, and the non-decomposable case when $\mathcal F_\mani$ is the cotangent bundle $\Omega^1_{\proj 2}$ of $\mathbb P^2$.  In the decomposable case, we show that it is neither projective, nor $\Pi$-projective. 
We then show instead that the second example coincides with the $\Pi$-projective plane $\mathbb P^{2}_\Pi$ introduced by Manin in \cite{Manin} in a completely 
different way. Actually, it can be shown that this unexpected correspondence is just a particular case of a general fact \cite{Noja}, showing that $\Pi$-projective geometry arises naturally as one considers the cotangent bundle of projective spaces as the 
fermionic bundle over $\proj n$. 
\

\noindent We conclude this section by briefly settling our conventions. \vspace{3pt}\\

\noindent \emph{Conventions.}  We will refer mainly to \cite{Manin}, \cite{ManinNC} and \cite{VoMaPe} for the main definitions and notions in supergeometry and we limit ourselves just to fix some notations and to point up what is particularly relevant in what follows. 

Given a supermanifold $\mani=(|\mani|,\stsheaf)$ we will call $\mathcal{J}_\mani$ the sheaf of ideals generated by all the (nilpotent) odd sections, hence a \emph{nilpotent sheaf}, and we denote $\manir=(|\mani|, \slantone{\stsheaf}{\mathcal{J}_\mani})$ the \emph{reduced manifold} underlying $\mani.$

The structure sheaf $\stsheaf $, the reduced structure sheaf $\stsheafred =\slantone{\stsheaf}{\mathcal{J}_\mani}$ and the nilpotent sheaf $\mathcal{J}_\mani$ fit together in a short exact sequence of $\mathcal{O}_\mani$-modules: 
\bear
\xymatrix@R=1.5pt{ 
0 \ar[rr] && \mathcal{J}_\mani \ar[rr] & &  \stsheaf \ar[rr] && \stsheafred \ar[rr] && 0.
 }
\eear
We call this short exact sequence the \emph{structural exact sequence} for the supermanifold $\mani.$

A very natural question that arises looking at the structural exact sequence above is whether it is split or not. 
When it is split, that is when there exists a morphism $\pi : \mani \rightarrow \manir$, satisfying\footnote{Denoting with $\iota$ the embedding of $\manir$ in $\mani$, we call $\iota^\sharp$ and $\pi^\sharp$ the induced homomorphisms on the structure sheaves of algebras.} 
$\pi \circ \iota = id_{\manir}$, splitting the structural exact sequence of $\mani$ as follows
\bear
\xymatrix@R=1.5pt{ 
0 \ar[rr] && \mathcal{J}_\mani \ar[rr] & &  \stsheaf  \ar[rr]_{\iota^\sharp} && \ar@{-->}@/_1.3pc/[ll]_{\pi^\sharp} \stsheafred \ar[rr] && 0,
 }
\eear
then we say that $\mani $ is a \emph{projected supermanifold}. 
In particular, the structure sheaf of a projected supermanifold is given by the direct sum $\stsheaf = \stsheafred \oplus \mathcal{J}_\mani$ and the structural exact sequence becomes
\bear
\xymatrix@R=1.5pt{ 
0 \ar[rr] && \mathcal{J}_\mani \ar[rr] & &  \stsheafred \oplus \mathcal{J}_\mani  \ar[rr] && \stsheafred \ar[rr] && 0.
 }
\eear
For any supermanifold $\mani$ with odd dimension $q$ and nilpotent sheaf $\mathcal{J}_\mani$, we have a 
$\mathcal{J}_\mani$-\emph{adic filtration on} $\stsheaf$ of length $q$
\bear
\mathcal{J}_\mani^0 \defeq \stsheaf \supset \mathcal{J}_\mani \supset \mathcal{J}_\mani^2 \supset \mathcal{J}_\mani^3 \supset \ldots \supset \mathcal{J}_\mani^{q} \supset \mathcal{J}_\mani^{q+1} = 0.
\eear
We call the superspace $\mbox{{Gr}}\, \mani \defeq (|\mani|, \mbox{{Gr}}\, \stsheaf)$, with
\bear
\mbox{{Gr}}\, \stsheaf  = \stsheafred \oplus \slantone{\mathcal{J}_\mani}{\mathcal{J}^2_{\mani}} \oplus \ldots \oplus \slantone{\mathcal{J}_\mani^{q-1}}{\mathcal{J}^q_{\mani}} \oplus {\mathcal{J}^q_\mani},
\eear
the \emph{split supermanifold associated to} $\mani$. Observe that a split supermanifold is automatically projected.

The quotient $\mathcal{F}_\mani  = \slantone{\mathcal{J}_\mani}{\mathcal{J}^2_\mani}$ is called the \emph{fermionic sheaf} of the supermanifold $\mani.$ Recall that for any $k\geq 0$ the sheaf $\slantone{\mathcal{J}^k_\mani}{\mathcal{J}^{k+1}_\mani}$ 
is canonically a $\stsheafred$ sheaf of modules. 

We say that $\mani$ is a \emph{split supermanifold} if it is isomorphic to its split associated supermanifold  $\mbox{{Gr}} \, \mani.$ 
\begin{remark} It is well known that a supermanifold $\mani$ with $\operatorname{rk}(\mathcal{F}_\mani)=2$ is split if and only it is projected. Indeed if $q=2$ one always has  $\mathcal{J}^j_{\mani}=0$ for any $j\geq 3$, from which it easily follows $(\mathcal{O}_\mani)_1=(\mathcal{J}_{\mani})_1\cong \mathcal{F}_\mani$, and, if $\mani$ is projected,  also that $(\mathcal{O}_\mani)_0=\stsheafred\oplus (\mathcal{J}_{\mani})_0=\stsheafred\oplus Sym^2\mathcal{F}_\mani$, where $(\mathcal{O}_\mani)_0$ and $(\mathcal{O}_\mani)_1$ are the even and odd part of the structure sheaf of the supermanifold respectively. \end{remark}
\noindent Finally, we conclude this section by presenting the most important (non-trivial) examples of complex split supermanifolds: 
\begin{example}[Complex Projective Superspaces] The complex projective superspace of dimensione $n|m$, denoted by $\proj {n|m}$, is the supermanifold defined by the pair $(\proj n , \mathcal{O}_{\proj {n|m}})$, having the ordinary complex projective 
space $\proj n $ as reduced manifold - that completely determines the topological aspects -, while its structure sheaf $\mathcal{O}_{\proj {n|m}}$ is given by  
\bear \label{fact}
\mathcal{O}_{\proj {n|m}} & = &\ \bigoplus_{k\,\rm{even}}\bigwedge^k\mathcal{O}_{\proj n}(-1)^{\oplus m}\oplus \bigoplus_{k\,\rm{odd}}\Pi \bigwedge^k\mathcal{O}_{\proj n}(-1)^{\oplus m} \nonumber \\
\eear
where we have inserted the symbol $\Pi$ as a reminder for the parity reversing. Note, in particular, that the term $k=0$ corresponds to the structure sheaf of the reduced manifold $\mathbb{P}^n.$\\
This expression for the structure sheaf makes it clear that $\proj{n|m} $ is canonically isomorphic to $\mbox{{Gr}}\, \proj{n|m}$ and the projection $\pi : \proj {n|m} \rightarrow \proj n$ embeds, at the level of the structure sheaves, 
$\mathcal{O}_{\proj n}$ into $\mathcal{O}_{\proj {n|m}}$.
\end{example}

%%%%%%%%%%%%%%%%%%%%%%%%%%%%%%%%%%%%%%%%%%%%%%%%%%%
%\section{An Obstruction to Projection}
\section{Obstruction to Splitting and Projective Spaces}
%Following \cite{Manin}, we will review some obstruction for a (complex) supermanifold to admit a projection on its reduced part. We will achieve this by singling out a cohomological invariant that detects whenever there exists an obstruction to split the structural 
%exact sequence attached to a particular supermanifold. \\
%In the case a supermanifold has odd dimension equal to one there cannot be any obstruction, as the following obvious theorem establishes. 
%\begin{theorem}[Supermanifolds of dimension $n|1$] \label{dim1} Let $\mani \defeq (|\mani|, \stsheaf)$ a (complex) supermanifold of odd dimension 1. Then $\mani$ is defined up to isomorphism by the pair $(\manir, \mathcal{F}_\mani)$.
%\end{theorem}
%\begin{proof} If the parity splitting reads $\stsheaf = \mathcal{O}_{\mani, 0}\, \oplus \, \mathcal{O}_{\mani, 1}$ and the odd dimension of the supermanifold is 1, then $\mathcal{J}^2_{\stsheaf} = 0$ and one naturally has that 
%$\mathcal{O}_{\mani, 1} \cong \mathcal{J}_\mani \cong \mathcal{F}_\mani$, a (locally free) sheaf of $\stsheafred$-modules of rank 1, having odd parity, and $\mathcal{O}_{\mani,0} \cong \stsheafred$, so there can't be any bosonisation extending 
%$\stsheafred$. 
%\end{proof}
%%%%%%%%%%%%%%%%%%%%%%%%%%%%%%%%%%
\noindent  Obstruction theory for complex supermanifolds has been first discussed in the seminal work of Green \cite{Green} and recently extended and applied to moduli spaces of super Riemann surfaces by Donagi and Witten in \cite{DonWit}. Here we will consider the first obstruction to splitting or projectedness  following \cite{Manin}.
Such obstruction  might appear in the case the odd dimension of the supermanifold $\mani$ is at least 2:
\begin{theorem}[Obstruction to Splitting/Projectedness] Let $\mani$ be a (complex) supermanifold of odd dimension 2. Then $\mani$ is projected (and hence split) if and only if the obstruction class $\omega_\mani$ is zero in 
$H^1 (\manir, \mathcal{T}_{\manir} \otimes Sym^2 \mathcal{F}_\mani)$. 
\end{theorem}
%\begin{proof} Using the same notation as in the above lemma, assume that $\omega_\mani $ is trivial. Then, there exist local splittings $\{ \pi_i \}_{i \in I}$ such that $\omega_{ij} \defeq \left ( \pi_i - \pi_j  \right) \big \lfloor_{\mathcal{U}_i \cap \mathcal{U}_j}$ 
%is a coboundary, that is $\omega_{ij } = \left ( \psi_i - \psi_j \right )\big \lfloor_{\mathcal{U}_i \cap \mathcal{U}_j}$, for some $\{ \psi_i \}_{i \in I }$ such that $\psi_i : \mathcal{O}_{\mathcal{U}_i, red } \rightarrow \mathcal{O}_{\mathcal{U}_i, 0}$. Then we can 
%define $\pi^\prime_{i} = \pi_i - \psi_i $ so that 
%\bear
%\nonumber \omega^\prime = \left ( \pi^\prime_i - \pi^\prime_j \right ) \big \lfloor_{\mathcal{U}_i \cap \mathcal{U}_j} = \left ( \psi_i - \psi_i - \psi_j + \psi_j \right )\big \lfloor_{\mathcal{U}_i \cap \mathcal{U}_j } = 0.
%\eear 
%This implies that $\pi^\prime_i = \pi^\prime_j $ on the intersections $\mathcal{U}_i \cap \mathcal{U}_j$, therefore (restoring the original notation) we have a \emph{global} homomorphism of sheaves of rings 
%$\pi^\sharp_0 : \mathcal{O}_{\manir} \rightarrow \mathcal{O}_{\mani, 0}$ such that $\pi^\sharp_0 \lfloor_{\mathcal{U}_i} = \pi_i$ and such that $\iota^\sharp_0 \circ \pi^\sharp_0 = id_{\manir}.$ \\
%Conversely, let $\mani$ be projected, that is let $\pi^\sharp_0 : \mathcal{O}_{\manir} \rightarrow \mathcal{O}_{\mani, 0} $ be a global homomorphism splitting (\ref{se2}), then it is enough to put $\pi_i \defeq (\pi_0^\sharp)\big \lfloor_{\mathcal{U}_i}$ 
%to get a collection of local splittings defining a trivial cocycle.
%\end{proof}
\noindent The theorem above offers a simple way to detect when a complex supermanifold having odd dimension 2 fails to be projected by means of a cohomological invariant that can be computed by ordinary algebraic geometric methods. 
The knowledge of $\omega_\mani $ for a supermanifold $\mani$ of dimension $n|2$ is a fundamental ingredient in the characterisation of the given supermanifold. 
\begin{theorem}[Supermanifolds of dimension $n|2$] \label{dim2} Let $\mani$ be a (complex) supermanifold of dimension $n|2$. Then $\mani$ is defined up to isomorphism by the triple 
$(\manir, \mathcal{F}_\mani, \omega_\mani)$ where
$\mathcal{F}_\mani$ is a rank $0|2$ sheaf of locally-free $\mathcal{O}_{\manir}$-modules, the fermionic sheaf, and $\omega_\mani \in H^1 (\manir, \mathcal{T}_{\manir} \otimes Sym^2 \mathcal{F}_\mani).$
\end{theorem}
\noindent The crucial issue of finding a set of invariants that completely characterises complex supermanifolds having odd dimension greater than 2 (up to isomorphisms) and given reduced complex manifold remains - to the best knowledge of the 
authors - still open and it will be addressed in a follow-up paper.

We now apply Theorem \ref{dim2}  to the case that underlying manifolds are ordinary projective spaces $\proj n$, our aim being to identify the obstructions to projectedness and therefore to single out all the non-projected 
supermanifolds of odd dimension $2$ having $\proj n$ as reduced space.

Since we are working over $\proj n$, if the fermionic sheaf $\mathcal{F}_\mani$ is a locally-free sheaf of $\mathcal{O}_{\proj n}$-module having dimension $0|2$, then it follows that there exists an isomorphism 
$Sym^2 \mathcal{F}_\mani \cong \mathcal{O}_{\proj n} (k)$ for some $k$, since all the invertible sheaves over $\proj n $ are of the form $\mathcal{O}_{\proj n} (k)$ for some $k$ (indeed $\mbox{Pic} (\proj n) \cong \mathbb{Z}$). \\
The basic tool to be exploited here in association with Theorem \ref{dim2} is the (twisted) Euler sequence for the tangent space over $\proj n$, that reads
\bear
\xymatrix{
0 \ar[r] & \mathcal{O}_{\proj n} (k) \ar[r] &  \mathcal{O}_{\proj n} (k+1)^{\oplus n+1}  \ar[r] & \mathcal{T}_{\proj {n}} (k) \ar[r]  & 0.
}  
\eear
We now examine super extensions of projective space $\proj n$ for every $n=1, 2, \ldots$.
\begin{itemize}
\item[$\mathbf{ n = 1 :}$] In the case of $\proj 1$, one has to study whenever $H^1 (\mathcal{T}_{\proj 1} \otimes Sym^2 \mathcal{F}_\mani) = H^1 (\mathcal{T}_{\proj 1} (k))$ is non-zero. This is easily achieved, since remembering that over $\proj 1$ 
one has $\mathcal{T}_{\proj 1} \cong \mathcal{O}_{\proj 1} (2)$, it amounts to find a $k$ such that $H^1 (\mathcal{O}_{\proj 1} (2+k)) \neq 0$. This is realised in the cases $k=-l \leq -4$, and one finds 
$
H^1 (\mathcal{O}_{\proj 1} (2-l)) \cong \mathbb{C}^{l-3}.
$
These cohomology groups have a well-known description: they are the $\mathbb{C}$-vector spaces with bases given by $\left\{\frac{1}{(X_0)^j (X_1)^{l-2-j}}\right\}_{j=1}^{l-3}$, where $X_0$ and $X_1$ are the homogeneous coordinates of $\proj 1$, 
see for example the proof of Theorem 5.1 Chapter III of \cite{Har}. As a result, the non-projected supermanifolds over $\proj 1$ 
are those such that $Sym^2 \mathcal{F}_{\mani} \cong \mathcal{O}_{\proj 1 } (-l)$ with $l\geq 4$. 
\item[$\mathbf{ n = 2 :}$] The case over $\proj 2$ is by far the most interesting, and - surprisingly - it has been forgotten by Manin, as he studies fermionic super-extensions over projective spaces in \cite{Manin}. Since over $\proj 2$ one has 
$H^1 (\mathcal{O}_{\proj 2} (k)) = H^1 (\mathcal{O}_{\proj 2} (k+1) ) = 0$, the long exact sequence in cohomology induced by the Euler short exact sequence splits in two exact sequences. The one we are concerned with reads  
\bear \nonumber
\xymatrix@R=1.5pt{
0 \ar[r] & H^1(\mathcal{T}_{\proj 2}(k)) \ar[r] & H^2(\mathcal{O}_{\proj 2}(k)) \ar[r] & H^2(\mathcal{O}_{\proj 2} (k+1))^{\oplus 3} \ar[r] & H^2 (\mathcal{T}_{\proj 2} (k)) \ar[r] &  0. 
}  
\eear
Now it is convenient to distinguish between three different sub-cases.
\begin{itemize}
\item[$k>-3$:] This is the easiest one, since $H^2 (\mathcal{O}_{\proj 2} (k)) = 0$, which implies that $H^1 (\mathcal{T}_{\proj 2} (k))$ is zero. 
\item[$k=-3$:] In this case we have that $H^2(\mathcal{O}_{\proj 2} (-2))^{\oplus 3} = 0,$ so we get an isomorphism
\bear
H^1(\mathcal{T}_{\proj 2}(-3)) \cong H^2(\mathcal{O}_{\proj 2} (-3)) \cong \set{C},
\eear
and, again, this cohomology group is generated by the cohomology class $[\frac{1}{X_0X_1X_2}]$ induced by the 2-cocycle defined by $\frac{1}{X_0X_1X_2}\in \Gamma(\mathcal{U}_0\cap \mathcal{U}_1\cap \mathcal{U}_2,\mathcal{O}_{\proj 2}(-3)),$ where $\mathcal{U}_i \defeq \{[X_0:X_1:X_2] \in \mathbb{P}^2: X_i \neq 0\}$, so that $\mathcal{U}_0\cap \mathcal{U}_2 \cap \mathcal{U}_2 = \{[X_0:X_1:X_2] \in \mathbb{P}^2: X_0 \neq 0,\, X_1 \neq 0,\, X_2 \neq 0\}.$ 
\item[$k <-3$:] In this case both $H^2 (\mathcal{O}_{\proj 2} (k)) $ and $H^2 (\mathcal{O}_{\proj 2} (k+1)) $ are non-zero. Therefore, this makes it a little bit harder to explicitly evaluate $H^{1} (\mathcal{T}_{\proj 2} (k) ) $ directly. 
However, this can be achieved upon using \emph{Bott formulas} (see for example \cite{Okonek}) that give the dimension of cohomology groups of the (twisted) cotangent bundles of projective spaces. First of all, we observe that using Serre duality one gets
$H^1 (\mathcal{T}_{\proj 2 } (k)) \cong H^1 (\mathcal{T}^\ast_{\proj 2} (-k -3))^\vee$.
In general, Bott formulas guarantee that $H^q (\bigwedge^{p} T^\ast_{\proj n} (k)) = 0$ if $q \neq n $ and $q,k \neq 0$. In our specific case we have $q=1, n = 2, p=1$ and $-k-3 <-6 $, therefore $
H^1 (\mathcal{T}_{\proj 2 } (k)) = 0.$ 
\end{itemize}
The above computation yields that the only non-projected supermanifold having underlying manifold isomorphic to $\proj 2$ will have a fermionic sheaf $\mathcal{F}_\mani$ such that $Sym^2 \mathcal{F}_\mani \cong \mathcal{O}_{\proj 2} (-3).$
\item[$\mathbf{ n > 2 :}$] In this case it is easy to conclude that $H^1 (\mathcal{T}_{\proj n} (k )) = 0$ since in the long exact sequence in cohomology this group sits between $H^1 (\mathcal{O}_{\proj n} (k+1))^{\oplus n+1}$ and 
$H^2 (\mathcal{O}_{\proj n} (k))$ and both of these groups are zero for every $k$ if $n >2$.
\end{itemize}

\noindent The above results allow us to classify the non-projected supermanifolds having $\proj 1$ or $\proj 2$ as reduced space. We start with the $\proj 1$ case, which directly relies on the fact that vector bundles over $\proj 1$ have no continuous 
moduli. Indeed one has the following result (for more details see \cite{NojaPhD}).
\begin{theorem} [Non-projected supermanifolds over $\proj 1$ with odd dimension $2$] Every non-projected supermanifold over $\proj 1$ having odd dimension equal to 2 is characterised up to isomorphism by a triple 
$(\proj 1, \mathcal{F}_\mani, \omega)$ where $\mathcal{F}_\mani$ is a rank $0|2$ sheaf of $\mathcal{O}_{\proj 1}$-modules such that $\mathcal{F}_\mani \cong \Pi \mathcal{O}_{\proj 1} (m) \oplus \Pi \mathcal{O}_{\proj 1} (n)$ with $m+n = -\ell$, 
$\ell\geq4$ and $\omega$ is a non-zero cohomology class $\omega \in H^1 (\mathcal{O}_{\proj 1} (2-\ell)).$
\end{theorem} 
\noindent We omit the proof of the result above, for which a reference is \cite{Manin} Chapter 4, \S 2.10, or \cite{NojaPhD} for a more detailed exposition.
\vskip2mm
\noindent The situation is rather more complicated over $\proj 2$, since locally-free sheaves of $\mathcal{O}_{\proj 2}$-modules do not in general split as direct sums of invertible sheaves, and they might have a moduli space. In order to provide a 
classification of all non-projected supermanifolds over $\proj 2$, among the possible choices for the fermionic bundle $\mathcal{F}_{\mani}$, one has to account for all the vector bundles 
$ \mathcal{F}_{\mani}$ whose second symmetric power\footnote{in the supersymmetric sense: remember $\mathcal{F}_{\mani}$ is seen as a rank $0|2$ locally-free sheaf of $\mathcal{O}_{\proj 2}$-modules, so that $Sym^2 (\mathcal{F}_{\mani})$ is a 
rank $1|0$ locally-free sheaf of $\mathcal{O}_{\proj 2}$-modules} is isomorphic to the canonical sheaf $\mathcal{O}_{\proj 2} (-3)$ of $\proj 2$. This fixes the first Chern class of $\mathcal{F}_{\mani}$, but it is yet not enough to uniquely fix a moduli 
space for these vector bundles, as we would need to fix their second Chern class as well. 

\section{Non-Projected Supermanifolds over $\proj 2$}

\noindent We have seen in the previous section that in fermionic dimension equal to $2$ the only way a supermanifold over $\proj 2$ can be non-projected is whenever its fermionic sheaf is such that $Sym^2 \mathcal{F}_\mani \cong \mathcal{O}_{\proj 2} (-3).$
\noindent To simplify notations, we give the following definition.
\begin{definition}[The Supermanifolds $\proj 2_\omega (\mathcal{F}_\mani)$] We  will denote with  $\proj 2_\omega (\mathcal{F}_\mani)$ any supermanifold represented by the triple $(\proj {2}, \mathcal{F}_\mani, \omega_\mani)$ where $\mathcal{F}_\mani$ is a locally-free sheaf of 
$\mathcal{O}_{\proj 2}$-modules of rank $0|2$ such that $Sym^2 \mathcal{F}_\mani \cong \mathcal{O}_{\proj 2} (-3)$ and such that $\omega_\mani$ is a cohomology class in $ H^1 (\mathcal{T}_{\proj 2} (-3)) \cong \mathbb{C}.$  
\end{definition} 

\noindent Working over $\proj 2$ leads to consider a set of homogeneous coordinates $[X_0 : X_1 : X_2]$ on $\proj 2$ and in turn the set of the affine coordinates and their algebras over the three open sets of the covering 
$\mathcal{U}\defeq \{ \mathcal{U}_0, \mathcal{U}_1, \mathcal{U}_2\}$ of $\proj 2$, where $\mathcal{U}_i \defeq \{ [X_0: X_1:X_2] \in \mathbb{P}^2 : X_i \neq 0\}.$ In particular, we will have the following  
\begin{align} \label{TransMod1}
& \mathcal{U}_0 \defeq \left \{ X_0 \neq 0 \right \} \quad \rightsquigarrow \quad z_{10} \,\mbox{mod}\,{\mathcal{J}^2_{\mathpzc{M}}}\defeq \frac{X_1}{X_0}, \qquad  z_{20} \,\mbox{mod}\,{\mathcal{J}^2_{\mathpzc{M}}} \defeq \frac{X_2}{X_0};   \nonumber \\
& \mathcal{U}_1 \defeq \left \{ X_1 \neq 0 \right \} \quad \rightsquigarrow \quad z_{11}\,\mbox{mod}\,{\mathcal{J}^2_{\mathpzc{M}}} \defeq \frac{X_0}{X_1}, \qquad z_{21} \,\mbox{mod}\,{\mathcal{J}^2_{\mathpzc{M}}}\defeq \frac{X_2}{X_1};   \nonumber \\  
& \mathcal{U}_2 \defeq \left \{ X_2 \neq 0 \right \} \quad \rightsquigarrow \quad z_{12} \,\mbox{mod}\,{\mathcal{J}^2_{\mathpzc{M}}} \defeq \frac{X_0}{X_2}, \qquad z_{22} \,\mbox{mod}\,{\mathcal{J}^2_{\mathpzc{M}}} \defeq \frac{X_1}{X_2}.  
\end{align}
The transition functions between these charts therefore look like
\begin{align}\label{transfzeta2}
\mathcal{U}_{0} \cap \mathcal{U}_{1} : \qquad & z_{10} \,\mbox{mod}\,{\mathcal{J}^2_{\mathpzc{M}}}= \frac{1}{z_{11}}\,\mbox{mod}\,{\mathcal{J}^2_{\mathpzc{M}}},  \quad & z_{20} &\,\mbox{mod}\,{\mathcal{J}^2_{\mathpzc{M}}}= \frac{z_{21} }{z_{11} } \,
\mbox{mod}\,{\mathcal{J}^2_{\mathpzc{M}}}; \nonumber   \\
\mathcal{U}_{0} \cap \mathcal{U}_{2} : \qquad & z_{10} \,\mbox{mod}\,{\mathcal{J}^2_{\mathpzc{M}}} = \frac{z_{22}}{z_{12}} \,\mbox{mod}\,{\mathcal{J}^2_{\mathpzc{M}}}, \quad & z_{20} & \,\mbox{mod}\,{\mathcal{J}^2_{\mathpzc{M}}}= \frac{1}{z_{12}}\,
\mbox{mod}\,{\mathcal{J}^2_{\mathpzc{M}}}; \nonumber \\
\mathcal{U}_{1} \cap \mathcal{U}_{2} : \qquad & z_{11} \,\mbox{mod}\,{\mathcal{J}^2_{\mathpzc{M}}} = \frac{z_{12}}{z_{22}} \,\mbox{mod}\,{\mathcal{J}^2_{\mathpzc{M}}}, & \quad z_{21} & \,\mbox{mod}\,{\mathcal{J}^2_{\mathpzc{M}}} = \frac{1}{z_{22}} \,
\mbox{mod}\,{\mathcal{J}^2_{\mathpzc{M}}}.
\end{align}
The reason why we give expressions for the local bosonic coordinates $z_{ij}$ and their transformation functions mod $\mathcal{J}^2_{\mani}$ instead of  mod $\mathcal{J}_{\mani}$ is that, as the fermionic dimension is equal to $2$, one has $(\mathcal{J}_{\mani})_0=\mathcal{J}^2_{\mani}$.

Moreover we will denote $\theta_{1i},\ \theta_{2i}$ a basis of the rank $0|2$ locally-free sheaf $\mathcal{F}_\mani$ on any of the open sets $\mathcal{U}_i$, for $i=0,1,2$, and, since $\mathcal{J}_{\mani}^3=0$, the transition functions among these bases will have the form
\begin{align}\label{transtheta}
\mathcal{U}_{i} \cap \mathcal{U}_{j} : \qquad &  \left(\begin{array}{l} \theta_{1i}\\ \theta_{2i}\end{array}\right)=M\cdot \left(\begin{array}{l} \theta_{1j}\\ \theta_{2j}\end{array}\right),
\end{align}
with $M$ a $2\times 2$ matrix with coefficients in $\mathcal{O}_{\proj 2}(\mathcal{U}_i\cap \mathcal{U}_j)$. Note that in the transformation (\ref{transtheta}) one can write $M$ as a matrix with coefficients given by some even rational functions of 
$z_{1j},z_{2j}$, because of the definitions (\ref{TransMod1}) and the facts that $\theta_{hj}\in \mathcal{J}_\mani$ and $\mathcal{J}^3_\mani=0$.

Finally we note the transformation law for the products $\theta_{1i}\theta_{2i}=(\det M)\theta_{1j}\theta_{2j}$, which, since $\det M$ is a transition function for the line bundle $Sym^2 \mathcal{F}_\mani \cong \mathcal{O}_{\proj 2} (-3)$ over 
$\mathcal{U}_{i} \cap \mathcal{U}_{j}$, can be written, up to constant changes of bases in $\mathcal{F}|_{\mathcal{U}_{i}}$ and $\mathcal{F}|_{\mathcal{U}_{j}}$, in the more precise form 
\begin{equation}\label{transtheta2}
\theta_{1i}\theta_{2i}=\left(\frac{X_j}{X_i}\right)^3\theta_{1j}\theta_{2j}.
\end{equation}
This also means that we can identify the base $\theta_{1i}\theta_{2i}$ of $Sym^2 \mathcal{F}_\mani|_{\mathcal{U}_i}$ with the standard base $\frac{1}{X_i^3}$ of $\mathcal{O}_{\proj 2} (-3)$ over $\mathcal{U}_{i}$.

The relations and transition functions given above are those that \emph{all} the supermanifolds of the kind $\proj 2_\omega (\mathcal{F}_\mani)$ share, regardless of the specific form of its fermionic sheaf $\mathcal{F}_\mani$. 

\begin{theorem}[Non-Projected Supermanifolds over $\proj 2$]\label{gentransition} Let $\mani$ be a supermanifold over $\proj 2$ having odd dimension equal to 2. Then $\mani$ is non-projected if and only if it arises from a triple 
$(\proj 2, \mathcal{F}_\mani, \omega)$ where $\mathcal{F}_\mani$ is a rank $0|2$ locally free sheaf of $\mathcal{O}_{\proj 2}$-modules such that $Sym^2 \mathcal{F}_{\mani} \cong \mathcal{O}_{\proj 2} (-3)$ and $\omega$ is a 
non-zero cohomology class $\omega \in H^1 (\mathcal{O}_{\proj 2} (-3)).$ 
One can write the transition functions for an element of the family $\proj 2_\omega (\mathcal{F}_\mani)$ from coordinates on $\mathcal{U}_0$ to coordinates on $\mathcal{U}_1$ as follows
\bear \label{eq:eventranf}
\left ( 
\begin{array}{c}
z_{10} \\
z_{20} \\
\theta_{10} \\
\theta_{20} 
\end{array}
\right ) = \left ( \begin{array}{c@{}}  
 \frac{1}{z_{11}} \\
 \frac{z_{21}}{z_{11}} + \lambda \frac{\theta_{11} \theta_{21}}{(z_{11})^2} \\
M  \left (
\begin{array}{c}
\theta_{11} \\
\theta_{21} 
\end{array}
\right ) 
\end{array}
\right ) = \Phi \left (z_{11}, z_{21}, \theta_{11}, \theta_{21}  \right )
\eear
where $\lambda \in \mathbb{C}$ is a representative of the class $\omega \in H^1 (\mathcal{T}_{\proj 2} (-3)) \cong \mathbb{C}$ and $M$ is a $2\times 2$ matrix with coefficients in $\mathbb{C}[z_{11}, z_{11}^{-1},z_{21}]$ such that $\det M=1/z_{11}^3$. 
Similar transformations hold between the other pairs of open sets.
\end{theorem}
\begin{proof} The part of the transformation law  (\ref{eq:eventranf}) that relates the fermionic coordinates $\theta_{10},\theta_{20}$ and $\theta_{11},\theta_{21}$ has already been discussed above. We are therefore left to explain the part of the transformation  
(\ref{eq:eventranf}) that relates the bosonic coordinates $z_{10},z_{20}$ and  $z_{11},z_{21}$. Writing the general transformation
\begin{equation}\label{eq:ztransf}
z_{\ell i} (\underline z_j , \underline \theta_j) = z_{\ell i} (\underline z_j) + \omega_{ij} (\underline z_j , \underline \theta_j)(z_{\ell i}) \qquad \ell = 1, \ldots, n,
\end{equation}
in this particular case, yields the following
\begin{eqnarray*}
z_{10}& = &\frac{1}{z_{11}} + \omega(z_{10})\\
z_{20} &=& \frac{z_{21}}{z_{11}} + \omega(z_{20}),
\end{eqnarray*}
with $\omega$ a derivation of $\mathcal{O}_{\proj 2}$ with values in $Sym^2\mathcal{F}_\mani$, which identifies an element  $\omega_{\mani}\in H^1 (\mathcal{T}_{\proj 2} \otimes Sym^2 \mathcal{F}_{\mani})$. Recall that by Theorem \ref{dim2} it is only 
the cohomology class $\omega_\mani$ that matters in defining the structure of the supermanifold $\mani$. In particular $\mani$ is non-projected if and only if  $\omega\in H^1 (\mathcal{T}_{\proj 2} \otimes Sym^2 \mathcal{F}_{\mani_{\, \proj 2}})$ is 
non-zero. The only possibility for this space to be non-zero is $Sym^2 \mathcal{F}_{\mani_{\, \proj 2}} \cong \mathcal{O}_{\proj 2} (-3)$, so that $\omega$ lies in $H^1 (\mathcal{T}_{\proj 2} (-3))$. Indeed this space is non-null as can be seen by the 
(twisted) Euler exact sequence for the tangent space, which reads
\bear \label{twisteul}
\xymatrix{
0 \ar[r] & \mathcal{O}_{\proj 2} (-3) \ar[r] &  \mathcal{O}_{\proj 2} (-2)^{\oplus 3}  \ar[r] & \mathcal{T}_{\proj {2}} (-3) \ar[r]  & 0.
}  
\eear
The long exact sequence in cohomology yields the following isomorphism:
\bear
\delta : H^1 (\mathcal{T}_{\proj 2} \otimes \mathcal{O}_{\proj 2} (-3)) \stackrel{\cong}\longrightarrow H^2 (\mathcal{O}_{\proj 2 }(-3) )\cong \mathbb{C}
\eear
where $\delta$ is the connecting homomorphism.  We now will make this isomorphism more explicit. Recall that the untwisted Euler sequence
is 
\bear
\xymatrix{
0 \ar[r] &  \mathcal{O}_{\proj 2}  \ar[r]^{e} & \mathcal{O}_{\proj 2} (1)^{\oplus 3}  \ar[r]^{\pi_\ast} & \mathcal{T}_{\proj {2}} (-3) \ar[r] & 0
}
\eear
where, if we  write formally $\mathcal{O}_{\proj 2} (1)^{\oplus 3}=\mathcal{O}_{\proj 2} (1)\partial_{X_0}\oplus \mathcal{O}_{\proj 2} (1)\partial_{X_1}\oplus \mathcal{O}_{\proj 2} (1)\partial_{X_2}$, we have 
\begin{align}
& e(f)=f(X_0\partial_{X_0}+X_1\partial_{X_1}+X_2\partial_{X_2}) & \\ 
& \pi_\ast(X_i\partial_{X_{j}}) =  \partial_{(X_j/X_i)}.&
\end{align}
The last relation takes place over the open set $\mathcal{U}_i$, with affine coordinates $X_j/X_i$, for $j \neq i$.
This holds because, fibrewise, the Euler sequence is provided by the differentials $\pi_\ast:T_{(\mathbb{C}^3)^\ast,v}\to T_{\proj 2, [v]}$ of the canonical projection $\pi: (\mathbb{C}^3)^\ast\to \proj 2$. In particular, over $\mathcal{U}_0$ we have the 
local splitting of $\mathcal{O}_\mani$ given by identifying $z_{10}=X_1/X_0$, $z_{20}=X_2/X_0$ and fermionic coordinates given by the chosen local base $\theta_{10},\theta_{20}$ of $\mathcal{F}$, and we get $\partial_{z_{20}}=\pi_\ast(X_0\partial_{X_2})$. 
By similar reasons we can write $\partial_{z_{11}}=\pi_\ast(X_1\partial_{X_0})$ over $\mathcal{U}_1$ and $\partial_{z_{22}}=\pi_\ast(X_2\partial_{X_1})$ over $\mathcal{U}_2$.
Now consider the local section $\frac{1}{X_0X_1X_2} \in \mathcal{O}_{\proj 2}(-3)(\mathcal{U}_0\cap \mathcal{U}_1\cap \mathcal{U}_2)$, whose class $[\frac{1}{X_0X_1X_2}]$ is a basis of $ H^2 (\mathcal{O}_{\proj 2}(-3))$. We make the following calculation 
on local sections over
$\mathcal{U}_0\cap \mathcal{U}_1\cap \mathcal{U}_2$   of the sequence (\ref{twisteul})
\begin{eqnarray*} e\left(\frac{1}{X_0X_1X_2}\right)%&=&\frac{X_0\partial_{X_0}+X_1\partial_{X_1}+X_2\partial_{X_2}}{X_0X_1X_2}\\
% &=& \frac{1}{X_0^3}\left( \frac{X_0}{X_1}\right)X_0\partial_{X_2}+\frac{1}{X_1^3}\left( \frac{X_1}{X_2}\right)X_1\partial_{X_0}+ \frac{1}{X_2^3}\left( \frac{X_2}{X_0}\right)X_2\partial_{X_1}\\
 &=& \theta_{10}\theta_{20}\left( \frac{X_0}{X_1}\right)X_0\partial_{X_2}+\theta_{11}\theta_{21}\left( \frac{X_1}{X_2}\right)X_1\partial_{X_0}+
 \theta_{12}\theta_{22}\left( \frac{X_2}{X_0}\right)X_2\partial_{X_1}.
\end{eqnarray*}
By applying $\pi_\ast$ to the last expression above we obtain 
\begin{eqnarray*}
0 &=& \frac{\theta_{10}\theta_{20}}{z_{10}}\partial_{z_{20}}+\frac{\theta_{11}\theta_{21}}{z_{21}}\partial_{z_{11}}+\frac{\theta_{12}\theta_{22}}{z_{12}}\partial_{z_{22}}\\
&=& \frac{\theta_{11}\theta_{21}}{z_{11}^2} \partial_{z_{20}}+\frac{\theta_{12}\theta_{22}}{z_{22}^2}\partial_{z_{11}}+\frac{\theta_{10}\theta_{20}}{z_{20}^2}\partial_{z_{22}}
\end{eqnarray*}
where, for the last equality, we have used the transformations (\ref{transfzeta2}) and (\ref{transtheta2}).
The final result is that the assignments of local sections of $\mathcal{T}_{\proj 2} \otimes Sym^2 \mathcal{F}$
\begin{eqnarray}\label{omegas}
\omega_{01}=\frac{\theta_{11}\theta_{21}}{z_{11}^2} \partial_{z_{20}} & \rm{on}  & \mathcal{U}_0\cap\mathcal{U}_1,\nonumber \\
\omega_{12}=\frac{\theta_{12}\theta_{22}}{z_{22}^2}\partial_{z_{11}}  & \rm{on}  & \mathcal{U}_1\cap\mathcal{U}_2,\nonumber \\
\omega_{20}=\frac{\theta_{10}\theta_{20}}{z_{20}^2}\partial_{z_{22}} & \rm{on}  & \mathcal{U}_0\cap\mathcal{U}_0 
\end{eqnarray}
satisfy the cocycle condition 
$$ \omega_{01}|_{\mathcal{U}_0\cap \mathcal{U}_1\cap \mathcal{U}_2}+\omega_{12}|_{\mathcal{U}_0\cap \mathcal{U}_1\cap \mathcal{U}_2}+\omega_{20}|_{\mathcal{U}_0\cap \mathcal{U}_1\cap \mathcal{U}_2}=0$$
and therefore define a cohomology class $[\omega]\in H^1(\mathcal{T}_{\proj 2} \otimes Sym^2 \mathcal{F}_{\mani})$.  Moreover, by definition of  the connecting homomorphism $\delta$, one has $$\delta([\omega])=\left[\frac{1}{X_0X_1X_2}\right]\in 
H^2(\mathcal{O}_{\proj 2}(-3)).$$ In particular, from the class $[\lambda \omega]\in H^1(\mathcal{T}_{\proj 2} \otimes Sym^2 \mathcal{F}_{\mani})$, one obtains the claimed transformation
\begin{eqnarray*}
z_{10}& = &\frac{1}{z_{11}} + \lambda\omega_{01}(z_{10})=\frac{1}{z_{11}},\\
z_{20} &=& \frac{z_{21}}{z_{11}} + \lambda\omega_{01}(z_{20})= \frac{z_{21}}{z_{11}}+\lambda \frac{\theta_{11}\theta_{21}}{z_{11}^2}.
\end{eqnarray*}
\end{proof}
\begin{remark} Before we go on, we recall that, as observed in \cite{Manin}, the class of isomorphism of the supermanifold does not depend on the parameter $\lambda$, since an isomorphism between the supermanifold corresponding to $\lambda =1$ and the supermanifold corresponding to an arbitrary $\lambda \neq 0$ is locally defined by $z^\prime_{ij} = z_{ij}$, $\theta^\prime_{1j} = \lambda \theta_{1j}$ and $\theta^\prime_{2j} = \theta_{2j}.$
\end{remark}

\section{$\proj 2_\omega (\mathcal{F}_\mani)$ is a Calabi-Yau supermanifold}

There is an important remark to be done here: when dealing with a non-projected supermanifold, no exact sequence comes to our help to compute the (super) Chern class of the supermanifold involved. Actually, one needs to carry out explicit 
computations, investigating the Berezinian of the change of coordinates of the cotangent sheaf among charts, defining the Berezinian sheaf of the supermanifold (we refer to the literature for the notion of Berezinian in super linear algebra and the related 
construction of the Berezinian sheaf of a supermanifold, see in particular \cite{Manin}, \cite{NCR}, \cite{Witten}). There is, however, a distinct class of supermanifolds such that their Berezinian is trivial in the sense specified by the following definition.
\begin{definition}[Calabi-Yau Supermanifold] We say that a supermanifold $\mani$ is a Calabi-Yau supermanifold if its Berezinian sheaf is trivial. In other words, $\mani$ is a Calabi-Yau supermanifold if $Ber_{\mani} \cong \mathcal{O}_{\mani}.$
\end{definition}
Calabi-Yau supermanifolds are indeed super analogs of the ordinary Calabi-Yau manifolds, a privileged class of varieties having trivial canonical sheaf. Indeed, we recall that the de Rham complex attached to a supermanifold is \emph{not bounded from above} - 
there is no top-form on a supermanifold! - and therefore it is in general critical to generalize the notion of integration over supermanifolds. Actually, the Berezinian sheaf is the only meaningful supersymmetric analog of the canonical sheaf, in that its sections 
transform as densities and they are the objects to call for when one looks for a measure for integration involving nilpotent bits - the so-called \emph{Berezin integral}. Calabi-Yau supermanifolds enter many constructions in theoretical physics (see, in 
particular \cite{Sethi}, \cite{VafaAga}, \cite{WittenCP34}, \cite{NeitzkeVafa}), but they have never really undergone a deep mathematical investigation, which we now begin here. We will indeed prove in the next theorem that \emph{all} of the supermanifolds over $\proj 2$ of type $\proj2_\omega (\mathcal{F}_\mani)$ are Calabi-Yau supermanifolds.
\begin{theorem} All of the supermanifolds of the kind $\proj {2}_\omega (\mathcal{F}_\mani)$ are Calabi-Yau supermanifolds regardless of the choice of $\mathcal{F}_\mani$ and $\omega$. That is
$
Ber_{\proj 2_{\omega} (\mathcal{F}_\mani)} \cong \mathcal{O}_{\proj 2_\omega (\mathcal{F}_\mani)}.
$
\end{theorem}
\begin{proof} We can work locally, considering transformations between $\mathcal{U}_0$ and $\mathcal{U}_1$. 
Then, using the results of the previous section, we can write the transition functions for an element of the family $\proj 2_\omega (\mathcal{F}_\mani)$ as in (\ref{eq:eventranf})
\begin{comment}
\bear
\left ( 
\begin{array}{c}
z_{10} \\
z_{20} \\
\theta_{10} \\
\theta_{20} 
\end{array}
\right ) = \left ( \begin{array}{c@{}}  
\frac{1}{z_{11}} \\
\frac{z_{21}}{z_{11}} + \lambda \frac{\theta_{11} \theta_{21}}{(z_{11})} \\
M  \left (
\begin{array}{c}
\theta_{11} \\
\theta_{21} 
\end{array}
\right ) 
\end{array}
\right ) = \Phi \left (z_{11}, z_{21}, \theta_{11}, \theta_{21}  \right )
\eear
\end{comment}
where $\lambda \in \mathbb{C}$ is a representative of the class $\omega \in H^1 (\mathcal{T}_{\proj 2} (-3)) \cong \mathbb{C}$. We can now compute the (super) Jacobian of this transformation, obtaining: 
\begin{align}
\mbox{Jac}\, (\Phi) = \left ( \begin{array}{c|c}  
A & B \\
\hline
C & D 
\end{array}
\right )
\end{align}
where one has
\begin{align}
& A = \left ( \begin{array}{ccc} 
- \frac{1}{(z_{11})^2} & & 0 \\
- \frac{z_{21}}{(z_{11})^2}- 2 \lambda \frac{\theta_{11} \theta_{21}}{(z_{11})^3} & &  \frac{1}{z_{11}}
\end{array}
\right ) \qquad \qquad 
& B & = \left ( \begin{array}{ccc} 
0 & & 0 \\
\lambda \frac{\theta_{21}}{(z_{11})^2} & & - \lambda \frac{\theta_{11}}{(z_{11})^2}
\end{array}
\right ) \nonumber \\  
& C = \left ( \begin{array}{c@{}cc@{}}
\partial_{z_{11}} M \left ( \begin{array}{c}
\theta_{11} \\
\theta_{21}
\end{array}  
\right ) & \; & 
\partial_{z_{21}}
M \left ( \begin{array}{c}
\theta_{11} \\
\theta_{21}
\end{array}
\right )
\end{array}
\right ) \qquad \qquad 
& D & = M.
\end{align}
A direct computation of the Berezinian of this Jacobian gives
%matrix  using the well-known formula $\mbox{Ber}\, (\mbox{Jac}\, (\Phi)) = \det (A- BD^{-1}C) \det D^{-1}$. We have that
%\begin{align} 
%A - BD^{-1}C = \left (
%\begin{array}{ccc}
%- \frac{1}{z_{11}} & & 0 \\
%\ast & & \frac{1}{z_{11}} - H
%\end{array} 
%\right ) \quad \mbox{where} \quad H = \frac{\lambda}{(z_{11})^2}\left ( \theta_{21}, - \theta_{11} \right ) M^{-1} \partial_{z_{21}} M \left ( \begin{array}{c} 
%\theta_{11} \\
%\theta_{21}
%\end{array}
%\right ).
%\end{align}
%Now, on the one hand we have that $\det M = \frac{1}{(z_{11})^3}$ and therefore $\partial_{z_{21}} \det M = 0$. On the other hand, denoting $\delta \defeq \partial_{z_{21}}$, we find by explicit computation that 
%\begin{align}
%H & = \frac{\lambda}{(z_{11})^2}\left ( \theta_{21}, - \theta_{11} \right ) \frac{1}{\det M} \left ( \begin{array}{cc}
%d & - b \\
%- c & a
%\end{array}
%\right ) \left ( \begin{array}{cc}
%\delta a & \delta b \\
%\delta c & \delta d
%\end{array} 
%\right )
%\left ( \begin{array}{c} 
%\theta_{11} \\
%\theta_{21}
%\end{array}
%\right ) = {\lambda z_{11}} \, \mbox{Tr}(M^{-1} \delta M) \,\theta_{11} \theta_{21}  
%\end{align}
%and upon noticing that $ 0 = \delta (\det M) = \mbox{Tr} (M^{-1} \delta M)$, we find that $H=0$. Therefore one has
\begin{align}
\mbox{Ber} \, (\mbox{Jac}\, (\Phi)) = -1 
\end{align}
which concludes the proof. 
\end{proof}
\begin{remark} Observe that in the setting of the above Theorem there is no orientation issue that prevents from making $-1$ into $+1$ in all the overlaps. 

\end{remark}
\begin{remark} 
Note that $\operatorname{Jac}(\Phi)^{T}$, when expressed in the coordinates $z_{10},z_{20},\theta_{10},\theta_{20}$ by means of the map $\Phi^{-1}$, is the transition matrix for the super tangent sheaf $\mathcal{T}_\mani$ from $\mathcal{U}_1$ to 
$\mathcal{U}_0$.

\noindent
Therefore the inverse matrix $(\operatorname{Jac}(\Phi)^{T})^{-1}$ is the transition matrix for $\mathcal{T}_\mani$ from $\mathcal{U}_0$ to $\mathcal{U}_1$.

\noindent
 In particular one has 
\bear 
(\operatorname{Jac}(\Phi)^T \mbox{ mod }\mathcal{J}_\mani)^{-1}
%& = & \left ( \begin{array}{c|c}  
%A^T\mbox{ mod }\mathcal{J} & 0 \\
%\hline
%0 & M^T 
%\end{array}
%\right )^{-1}\\
&=& \left ( \begin{array}{c|c}  
(A^T)^{-1} \mbox{ mod }\mathcal{J}& 0 \\
\hline
0 & (M^T)^{-1} 
\end{array}
\right )
\eear
Note that $(A^T)^{-1} \mbox{ mod }\mathcal{J}_\mani$ is the transition matrix for $\mathcal{T}_{\proj 2}$ and $(M^T)^{-1}$ is the one for $\mathcal{F}^\ast_\mani$, therefore  one obtains the isomorphism
\begin{equation}\label{restrTan}
\mathcal{T}_\mani \lfloor_{\proj 2}\cong \mathcal{T}_{\proj 2}\oplus \mathcal{F}_\mani^\ast.
\end{equation}
\end{remark}

\begin{remark} We recall that in \cite{Manin} the Berezinian sheaf of a supermanifold is defined as $Ber_\mani \defeq Ber((\Omega_\mani^1)_{\operatorname{odd}})^\ast$,
where $(\Omega_\mani^1)_{\operatorname{odd}}$ is the cotangent sheaf of $\mani$ defined by means of the {\em odd} differentiation $f\mapsto d_{\operatorname{odd}}f$. This is very convenient for example  for the integration theory on a supermanifold. 
We have adopted the equivalent definition $Ber_\mani \defeq Ber(\mathcal{T}_\mani)$, with $\mathcal{T}_\mani$ the tangent sheaf, stressing the similarity of this definition with the definition of the first Chern class $c_1(M)=c_1(\mathcal{T}_M)$ of an ordinary manifold $M$. 
The fact that the two definitions agree can be seen by comparing the Berezinians of the related transition matrices of $\mathcal{T}_\mani$ and  $(\Omega_\mani^1)_{\operatorname{odd}}$.
\end{remark}
\begin{remark}As concluding remark, we stress that it is likely that this result holds true for a larger class of (non-projected) supermanifolds, having $\manir$ given by any complex surface and $\mathcal{E}$ a rank two locally free sheaf such that  
$\det \mathcal{E} = K_{\manir}$, the canonical line bundle of the reduced manifold. We leave to a future work to provide a suitable generalisation of adjunction theory applicable in the case of non-projected supermanifolds.
\end{remark}
\section{Even Picard Group of $\proj2_{\omega} (\mathcal{F}_\mani)$ and non-projectivity}

Now that we know some geometry of the non-projected Calabi-Yau supermanifolds $\proj 2_\omega (\mathcal{F}_\mani)$, we take on the main focus of the paper: whether there exists an embedding of these supermanifolds into some 
supermanifold {\em with a universal property}, such as, for example, a projective superspace $\proj {n|m}$. 
In the language of Grothendieck, this calls for a search for \emph{very ample} locally-free sheaves of $\mathcal{O}_\mani$-modules of rank $1|0$. \\
The invertible sheaf $\mathcal{O}_{\proj {n|m}} (1)$ on $\proj {n|m}$, defined as the pull-back $\pi^\ast\mathcal{O}_{\proj n}(1)$ by the canonical projection $\pi\colon \proj {n|m}\to \proj n$, plays an important role for maps from complex supermanifolds $\mani$ into  $\proj {n|m}$. Indeed, as in ordinary algebraic geometry, if $\mathcal{E}$ 
is a certain globally generated locally free sheaf of $\mathcal{O}_\mani$-modules of rank $1|0$, having $n+1|m$ global sections $\{ s_0, \ldots, s_n | \xi_1, \ldots, \xi_m \}$, then there exists a morphism $\phi_\mathcal{E} : \mani \rightarrow \proj{n|m}$ such that 
$\mathcal{E} = \phi^\ast_{\mathcal{E}} (\mathcal{O}_{\proj {n|m}} (1))$ and such that $s_i = \phi^\ast_{\mathcal{E}} (X_i) $ and $\xi_j = \phi^\ast_{\mathcal{E}} (\Theta_j)$ for $i = 0, \ldots, n$ and $j = 1, \ldots, m$. Notice that also the converse is true, that 
is given a morphism $\phi : \mani \rightarrow \proj {n|m} $, then there exists a globally generated sheaf of $\mathcal{O}_\mani$-modules $\mathcal{E}_\phi$ such that it is generated by the global sections $\phi^\ast (X_i)$ and $\phi^\ast (\Theta_j)$ for 
$i = 0, \ldots, n$ and $j = 1, \ldots, m$. Relying on this result, we can give the following definition.
\begin{definition}[Projective Supermanifold] We say that a complex supermanifold $\mani$ is projective if there exists a morphism $\phi : \mani \rightarrow \proj{n|m}$ such that $\phi$ is injective on $\manir$ and its differential $d\phi$ is injective everywhere 
on $\mathcal{T}_{\mani}$. Equivalently, $\phi$ identifies $\mani$ with a closed sub-supermanifold of $\proj {n|m}$.
\end{definition}
\noindent Clearly, again, this calls for a search for suitable (very ample) locally-free sheaves of $\mathcal{O}_\mani$-modules of rank $1|0$ to set up the morphism into $\proj {n|m}$. We will prove that for any non-projected supermanifold 
$\proj 2_\omega (\mathcal{F}_\mani)$, regardless of how one chooses the fermionic sheaf $\mathcal{F}_\mani$, such morphism cannot exist as there are no $1|0$ sheaves available to realize it. \\
Notice that considering a rank $1|0$ locally-free sheaf of $\mathcal{O}_\mani$-modules, that we will call an \emph{even invertible sheaf}, corresponds to deal with transition functions $g_{ij}$ taking values into $\left (\mathcal{O}_{\mani}^\ast \right )_0 \cong \mathcal{O}^\ast_{\mani,0}$ as the transformation needs to be invertible and a parity-preserving one. This has the important consequence that $\mathcal{O}^\ast_{\mani,0}$ is a sheaf of \emph{abelian groups}, so that we are allowed to consider its cohomology groups, without confronting the issues related to the definition of non-abelian cohomology (notice that the full sheaf $\mathcal{O}^\ast_\mani$ is indeed \emph{not} a sheaf of abelian groups). Clearly, in order to define an even invertible sheaf, the transition functions have to be $1$-cocycles valued in the sheaf $\mathcal{O}^\ast_{\mani,0}$,  therefore, calling \emph{even Picard group} $\mbox{Pic}_0 (\mani)$ the group of isomorphism classes of even invertible sheaves on a supermanifold $\mani$, one finds that 
\bear
\mbox{{Pic}}_0 (\mani) \cong H^1 (\mathcal{O}^\ast_{\mani, 0}).
\eear
In what follows, we will grant ourselves the liberty to call the cohomology group $H^1 (\mathcal{O}^\ast_{\mani, 0})$ the even Picard group of the supermanifold, by implicitly referring to the above isomorphism. \\
Clearly, an empty even Picard group $\mbox{Pic}_0 (\mani)$ is enough to guarantee the non-existence of the embedding into projective super space $\phi : \mani \rightarrow \proj {n|m}$.   
\begin{theorem}[$\proj2_\omega (\mathcal{F}_\mani) $ is non-projective]\label{thm:nonproj}The even Picard group of the non-projected supermanifold $\proj 2_\omega (\mathcal{F}_\mani)$ is trivial, regardless of how one chooses the fermionic sheaf $\mathcal{F}_\mani$: 
\bear
\mbox{\emph{Pic}}_0 (\proj 2_{\omega} (\mathcal{F}_\mani)) = 0. 
\eear 
In particular, for any $\mathcal{F}_\mani$ such that $Sym^2 \mathcal{F}_\mani \cong \mathcal{O}_{\proj 2} (-3)$, if $\proj2_\omega (\mathcal{F}_\mani)$ is  non-projected, then it is also non-projective.
\end{theorem}
\begin{proof} We put $\mani \defeq \proj 2_\omega (\mathcal{F}_\mani) $ and, remembering that $Sym^2 \mathcal{F}_\mani \cong \mathcal{O}_{\proj 2} (-3)$, we consider the short exact sequence 
\bear \label{starshort}
\xymatrix@R=1.5pt{
0 \ar[rr] & & \mathcal{O}_{\proj 2} (-3) \ar[rr]^{\rm exp} &&  \mathcal{O}^\ast_{\mani, 0}  \ar[rr] & & \mathcal{O^\ast}_{\proj 2} \ar[rr] && 1. 
}  
\eear
This is the multiplicative version of the structural exact sequence, with the first map defined by $\operatorname{exp}(h)=1+h$, as $h\in \mathcal{O}_{\proj 2} (-3)=(\mathcal{J}_\mani)_0$ and $(\mathcal{J}_\mani)_0^2=0$. The exact sequence above gives 
the following piece of long exact cohomology sequence 
\bear
\xymatrix@R=1.5pt{
0 \ar[r] & H^1(\mathcal{O}^\ast_{\mani, 0}) \ar[r] & H^1 (\mathcal{O}^\ast_{\proj 2})  \ar[r]^{\delta \quad} & H^2 (\mathcal{O}_{\proj 2} (-3)) \ar[r] &  \cdots .
}  
\eear
Now, one has $\mbox{Pic} (\proj 2) = H^1 (\mathcal{O}^\ast_{\proj 2})\cong \mathbb{Z}$ and $H^2 (\mathcal{O}_{\proj 2} (-3))\cong \mathbb{C}$, so everything reduces to decide whether the connecting homomorphism 
$\delta : \mbox{Pic} (\proj 2) \rightarrow H^2 (\mathcal{O}_{\proj 2} (-3))$ is the zero map or it is an injective map $\mathbb{Z}\to \mathbb{C}$. This can be checked directly, by a diagram-chasing computation, by looking at the following diagram of 
cochain complexes,
\bear
\xymatrix{
{C}^2 (\mathcal{O}_{\proj 2} (-3) ) \; \ar[r]^i & {C}^2 (\mathcal{O}^\ast_{\mani, 0}) &  \\ 
& {C}^1 (\mathcal{O}^\ast_{\mani, 0}) \ar[u] \ar@{->>}[r]^j & {C}^1 (\mathcal{O}^\ast_{\proj 2}), 
}  
\eear
that is obtained by considering the short exact sequence (\ref{starshort}) and the \v{C}ech cochain complexes of the sheaves involved in the sequence.\\ 
One then picks the generating line bundle $\langle \mathcal{O}_{\proj 2} (1) \rangle_{\mathcal{O}_{\proj 2}} \cong \mbox{Pic}(\proj 2)$ and, given the usual covering $\mathcal{U} \defeq \{ \mathcal{U}_{i}\}_{i=0}^2$ of $\proj 2$ as above, 
$\mathcal{O}_{\proj 2} (1)$ can be represented by the cocycle $g_{ij} \in Z^1 (\mathcal{U}, \mathcal{O}^\ast_{\proj 2})$ given by the transition functions of the line bundle itself. Explicitly, in homogeneous coordinates, these cocycles are given by 
\bear
\mathcal{O}_{\proj 2} (1) \longleftrightarrow \left \{ g_{01} = \frac{X_0}{X_1},\  g_{12} = \frac{X_1}{X_2},\  g_{20} = \frac{X_2}{X_0} \right \}.
\eear 
Since the map $j : {C}^1 (\mathcal{O}^\ast_{\mani, 0}) \rightarrow {C}^1 (\mathcal{O}^\ast_{\proj 2})$ is surjective, these cocycles are, in particular, images of elements in ${C}^1 (\mathcal{O}^\ast_{\mani, 0})$.
More precisely we have 
$$j(z_{11})=g_{01},\quad j(z_{22})=g_{12},\quad j(z_{20})=g_{20},$$ 
hence we can consider the lifting $\sigma=\{z_{11}, z_{22}, z_{20}\}$ of $\{g_{01},g_{12}, g_{20}\}$ to  ${C}^1 (\mathcal{O}^\ast_{\mani, 0})$. 
We stress  that  this is \emph{not} a cocycle in ${C}^1 (\mathcal{O}^\ast_{\mani, 0})$. Now, by going up in the diagram to ${C}^2 (\mathcal{O}^\ast_{\mani, 0})$ by means of the \v{C}ech boundary map 
$\delta : {C}^1 (\mathcal{O}^\ast_{\mani, 0})  \rightarrow {C}^2 (\mathcal{O}^\ast_{\mani, 0})  $, and using the bosonic transformation laws induced by the derivations (\ref{omegas}), one finds the following element: 
\begin{eqnarray*}
\delta(\sigma ) &= & 
%z_{11}\cdot z_{22}\cdot z_{20} \lfloor_{\mathcal{U}_0 \cap \mathcal{U}_1 \cap \mathcal{U}_2}\\
%&=& \left(\frac{z_{12}}{z_{22}}+\lambda\frac{\theta_{12}\theta_{22}}{z_{22}^2}\right)z_{22}z_{20} \\
%&=& 1+\lambda\frac{\theta_{12}\theta_{22}}{z_{22}}z_{20}=1+\lambda\theta_{12}\theta_{22}\left(\frac{X_2}{X_1}\right)\left(\frac{X_2}{X_0}\right)\quad {\rm by\ (\ref{TransMod1})}\\
%&=& 1+\lambda\left(\frac{1}{X_2^3}\right)\left(\frac{X_2}{X_1}\right)\left(\frac{X_2}{X_0}\right)=
1 + \frac{\lambda}{X_0 X_1 X_2}.  
\end{eqnarray*}   
We have that the element $1 + \frac{\lambda}{X_0 X_1 X_2}$ is the image of $\frac{\lambda}{X_0 X_1 X_2}$ through 
the map $i$. Hence we find that $\delta:\operatorname{Pic} (\proj 2) \longrightarrow H^2 (\mathcal{O}_{\proj 2} (-3))$ maps $ [\mathcal{O}_{\proj 2} (1)] \longmapsto [\frac{\lambda}{X_0 X_1 X_2}]$, which is non-zero for $\lambda\neq0$, i.e. for $\mani$ 
non-projected.
This leads to the conclusion that $\mbox{Pic}_0 (\proj 2_\omega (\mathcal{F}_\mani)) = H^1 (\mathcal{O}^\ast_{\proj 2_\omega (\mathcal{F}_\mani), 0}) = 0$, i.e. the only locally-free sheaf of rank $1|0$ on $\mani$ is $\mathcal{O}_\mani$. 

In particular there are no locally-free sheaves of rank $1|0$ to realise an embedding in a projective superspace, that is the non-projected supermanifold $\proj 2_\omega (\mathcal{F}_\mani)$ is non-projective.  
\end{proof}
\begin{remark} The previous theorem illustrates a substantial difference between complex algebraic supergeometry and the usual complex algebraic geometry, where projective spaces are the prominent 
ambient spaces. 
This fact was already known by Manin (see, for example \cite{Manin}, \cite{ManinNC}), who produced many examples of non-projective supermanifolds. 
However, in the next section \ref{sec:embedding} we will show that any $\mani=\proj 2_\omega(\mathcal{F}_\mani)$ can always be embedded in some super Grassmannian.
\end{remark}

\section{Embedding $\proj 2_\omega(\mathcal{F})$ into Super Grassmannians}\label{sec:embedding}
In this section we refer to \cite{Manin}, Chapter 4, \S 3, for a thorough treatment of super Grassmannians, and also of ordinary Grassmannians. A detailed review of the basic properties of super Grassmannians, with emphasis on their coordinate charts description, has recently appeared in \cite{NPG}.
\subsection{The universal property of super Grassmannians}
The super Grassmannians $G= G(a|b, V)$ have the following universal property. \vspace{5pt}\\
{\bf Universal Property:} for any superscheme $\mani$ and any locally-free sheaf of $\stsheaf$-modules $\mathcal{E}$ of rank $a|b$ on $\mani$ and any vector superspace $V\cong \mathbb{C}^{n|m}$ with a surjective sheaf map
 $V\otimes \mathcal{O}_\mani \to \mathcal{E}$, then there exists a {\em unique} map $\Phi:\mani\to G(a|b, V)$ such that the inclusion 
$\mathcal{E}^\ast \to V^\ast\otimes \mathcal{O}_\mani$ is the pull-back of the inclusion $\mathcal{S}_G\to \mathcal{O}_G^{\oplus n|m}$ from the sequence 
\bear \label{tautsuper}
\xymatrix@R=1.5pt{
0 \ar[rr] && \mathcal{S}_G \ar[rr] && \mathcal{O}_{G}^{\oplus n|m} \ar[rr] && \widetilde{\mathcal{S}}_G^\ast \ar[rr] &&  0.
}  
\eear
where $\mathcal{S}_G$ is the tautological sheaf of the super Grassmannian.\\ 
In this case, once a local basis $\{ e_1,\ldots,e_a | f_1,\ldots,f_b \}$ is fixed for $\mathcal{E}$ over some open set $\mathcal{U}$, then, over $\mathcal{U}$, the evaluation map $V\otimes \mathcal{O}_\mani\to \mathcal{E}$ is defined by a $(a|b)\times(n|m)$ 
matrix $M_\mathcal{U}$ with coefficients in $\mathcal{O}_\mani(\mathcal{U})$, and any reduction of  $M_\mathcal{U}$ into a standard form of type 
\begin{equation}\label{matrixstform}
\mathcal{Z}_{{I}} \defeq 
\left (
\begin{array}{ccc|ccc||ccc|ccc}
& &  & 1 \; & & & & & & & & \\
\;  & \; x_I \; & \; &  & \ddots & & & 0 & &\;  &\; \xi_I \; &\;  \\
& & & & & \; 1& & & & & & \\
\hline \hline
& & & & & & 1\; & & & & & \\
&\; \xi_I \; & & & 0& & &\ddots & & &\;  x_I \; & \\
& & & & & & & &\; 1 & & & 
\end{array}
\right ),
\end{equation}
by means of elementary row operations, is a local representation of the map $\Phi$.
\subsection{The Embedding Theorem}
We will prove the following result.
\begin{theorem}\label{embeddingth} Let $\mani = \proj 2_\omega(\mathcal{F}_\mani)$ and $\mathcal{T}_\mani$ its tangent sheaf. 
Let $V = H^0(Sym^k\mathcal{T}_\mani)$. Then, for any $k\gg 0$ the evaluation map $V\otimes\mathcal{O}_\mani\to Sym^k\mathcal{T}_\mani$ induces an embedding $\Phi_k:\mani\longrightarrow  G(2k|2k, V)$.
\end{theorem} 
\noindent We first introduce some notations.\vspace{6pt}\\
\noindent \emph{Notation.}  Having at our disposal the structure sheaf $\stsheaf$ of $\mani$ we can also consider the sub superscheme of $\mani$, given by the pair $(\proj 2, \mathcal{O}_{\mani^{(2)}})$, where we have defined 
$\mathcal{O}_{\mani^{(2)}} \defeq \slantone{\mathcal{O}_{\mani}}{\mathcal{J}_\mani^2}$. We stress that this is \emph{not} a supermanifold: indeed it fails to be locally isomorphic to any local model of the kind $\mathbb{C}^{p|q}$, and, more generally, 
it is locally isomorphic to an affine superscheme for some super ring. \\
We call this sub superscheme $\mani^{(2)}$ and we characterise its geometry in the following lemma.
\begin{lemma}[The Superscheme $\mani^{(2)}$]\label{lemM2} Let $\mani^{(2)}$ be the superscheme characterised by the pair $(\proj 2, \mathcal{O}_{\mani^{(2)}})$, where $\mathcal{O}_{\mani^{(2)}}= \slantone{\stsheaf}{\mathcal{J}_\mani^2}$. 
Then $\mani^{(2)} $ is projected and its structure sheaf, as a sheaf of $\mathcal{O}_{\proj 2}$-algebras, is $\mathcal{O}_{\mani^{(2)}} \cong \mathcal{O}_{\proj 2} \oplus \mathcal{F}_\mani$ .
\end{lemma}
\begin{proof} It is enough to observe that the parity splitting of the structure sheaf reads $\mathcal{O}_{\mani^{(2)}} = \slantone{\mathcal{O}_{\mani,0}}{\mathcal{J}_\mani^2} \oplus \slantone{\mathcal{O}_{\mani,1 }}{\mathcal{J}_{\mani}^2}$, hence the defining 
short exact sequence for the even part reduces to an isomorphism $\mathcal{O}_{\mani,0}^{(2)} \cong \mathcal{O}_{\proj 2}$. We therefore must have that the structure sheaf gets endowed with a structure of 
$\mathcal{O}_{\proj 2}$-module given by $\mathcal{O}_{\proj 2} \oplus \mathcal{F}_\mani$, that actually coincides with the parity splitting.
We observe that in the $\mathcal{O}_{\proj 2}$-algebra  $\mathcal{O}_{\mani^{(2)}}\cong \mathcal{O}_{\proj 2} \oplus \mathcal{F}_\mani$ the product $\mathcal{F}_\mani \otimes_{\mathcal{O}_{\proj 2}} \mathcal{F}_\mani \rightarrow \mathcal{O}_{\proj 2}$ 
is null. \end{proof}
\noindent The first requirement for Theorem \ref{embeddingth} is that the morphism $\Phi_k$ is well defined. This is a consequence of the following Lemma.
\begin{lemma} The following facts hold.
\begin{enumerate}
\item The restriction maps $V\longrightarrow H^0(Sym^k \mathcal{T}_\mani|_{\mani^{(2)}})$ and $V\longrightarrow H^0(Sym^k \mathcal{T}_\mani|_{\proj 2})$ are surjective for $k\gg 0$.
\item
The locally-free sheaf of $\stsheaf$-modules $Sym^k\mathcal{T}_\mani$ is generated by global sections, i.e. the evaluation map
$V\otimes \mathcal{O}_\mani\to  Sym^k\mathcal{T}_\mani$ is surjective, for $k\gg 0$.
\end{enumerate}
\end{lemma}
\begin{proof} Let us consider the composition of linear maps
\begin{equation}\label{surjglobsec} V\longrightarrow H^0(Sym^k \mathcal{T}_\mani |_{\mani^{(2)}})\longrightarrow H^0(Sym^k\mathcal{T}_\mani |_{\proj 2})\longrightarrow Sym^k\mathcal{T}_\mani(x),
\end{equation}
with $Sym^k\mathcal{T}_\mani(x)$ the fibre at $x$. By the supercommutative version of the Nakayama Lemma - see for example lemma 4.7.1 in \cite{Vara} - to prove fact (2) one has to show that for any $x\in {\proj 2}$ the linear map 
$V\to Sym^k\mathcal{T}_\mani(x)$ is surjective. 
Therefore we can reduce ourselves to show the surjectivity of all the linear maps in the composition, which will also include a proof of fact (1). 
 For simplicity of notation we set $\mathcal{E}_k \defeq Sym^k\mathcal{T}_\mani$. To see the surjectivity of the last map, we observe that, by (\ref{restrTan}), one has
 \begin{eqnarray}\label{Skdecomp}
 \mathcal{E}_k|_{\proj 2}&\cong &Sym^k( \mathcal{T}_{\proj 2}\oplus \mathcal{F}^\ast_\mani) \nonumber \\
 &=&(Sym^k\mathcal{T}_{\proj 2})\oplus (Sym^{k-1} \mathcal{T}_{\proj 2}\otimes \mathcal{F}^\ast_\mani)\oplus  (Sym^{k-2} \mathcal{T}_{\proj 2}\otimes Sym^2 \mathcal{F}^\ast_\mani),
 \end{eqnarray}
 all the other summands being $0$, since $\mathcal{F}_\mani =\Pi E$ for some vector bundle $E$ of rank 2 and $Sym^{i}\mathcal{F}_\mani =\Pi^{i}\bigwedge^i E$. \\
 Now one can use the well known ampleness of the vector bundle $\mathcal{T}_{\proj 2}$ (see \cite{Hartample} for the definition of an ample vector bundle 
 in algebraic geometry) to conclude that all the higher cohomology groups $H^i( Sym^k\mathcal{T}_{\proj 2}(-i))$, $
H^i (Sym^{k-1} \mathcal{T}_{\proj 2}\otimes \mathcal{F}^\ast_\mani(-i))$,  $H^i (Sym^{k-2} \mathcal{T}_{\proj 2}\otimes Sym^2 \mathcal{F}^\ast_\mani(-i))$ vanish for $k\gg 0$, and hence all these vector bundles are generated by global sections, since they have Castelnuovo-Mumford regularity index equal to $0$, see \cite{Mumford}, lecture 14.

Alternatively, one can use the exact sequences 
\bear\label{SymmEuler}
\xymatrix@R=1.5pt{
0 \ar[r] &  (Sym^{m-1}\mathcal{O}_{\proj 2}^{\oplus 3})(m-1) \ar[r] & (Sym^m \mathcal{O}^{\oplus 3}_{\proj 2})(m)  \ar[r] & Sym^m \mathcal{T}_{\proj 2} \ar[r] & 0, 
}  
\eear
deduced from the Euler sequence, tensor them with $Sym^j \mathcal{F}_\mani$ for $j=0,1,2$ and use the 
fact that $H^i(\mathcal{F}_\mani(m))=0$ for any $i>0$ and that $\mathcal{F}_\mani(m)$ is generated by global sections, for any $m\gg 0$, to deduce the same conclusions for $Sym^m \mathcal{T}_{\proj 2}\otimes Sym^j\mathcal{F}_\mani$.

Recall the exact sequence 
\bear
\xymatrix@R=1.5pt{
0 \ar[r] &  \mathcal{E}_k\otimes \mathcal{J}_\mani \ar[r] & \mathcal{E}_k \ar[r] &  \mathcal{E}_k\lfloor_{\proj 2} \ar[r] & 0, 
}  
\eear
and observe that, as $\mathcal{J}_\mani^3=0$, one has that $\mathcal{J}_\mani$ is a 
$\mathcal{O}_\mani/\mathcal{J}_\mani^2$-module, \emph{i.e.} a $\mathcal{O}_{\mani^{(2)}}$-module.
As such, by Lemma \ref{lemM2} one also knows that $\mathcal{J}_\mani$, and hence also $\mathcal{E}_k\otimes \mathcal{J}_\mani$, has a structure of a $\mathcal{O}_{\proj 2}$-module, given as 
$\mathcal{E}_k\otimes \mathcal{J}_\mani \cong (\mathcal{E}_k |_{\proj 2} \otimes Sym^2\mathcal{F}_\mani)\oplus (\mathcal{E}_k |_{\proj 2}\otimes \mathcal{F}_\mani)\cong (\mathcal{E}_k |_{\proj 2}(-3))\oplus (\mathcal{E}_k|_{\proj 2}\otimes \mathcal{F}_\mani) $. \\
Similarly, let us consider the exact sequence 
\bear
\xymatrix@R=1.5pt{
0 \ar[r] &  \mathcal{E}_k\otimes \mathcal{J}_\mani^2 \ar[r] & \mathcal{E}_k \ar[r] & \mathcal{E}_k \lfloor_{\mani^{(2)}} \ar[r] & 0, 
}  
\eear
where $\mathcal{E}_k\otimes \mathcal{J}_\mani^2\cong \mathcal{E}_k |_{\proj 2}(-3)$ is a 
$\mathcal{O}_{\proj 2}$-module. Similarly as above, one can show that $H^1(\mathcal{E}_k |_{\proj 2}\otimes \mathcal{F}_\mani)=0$ and $H^1(\mathcal{E}_k |_{\proj 2}(-3))=0$ for $k\gg 0$, hence one has that 
$H^0(\mathcal{E}_k)\to H^0(\mathcal{E}_k|_{\mani^{(2)}})$ and $H^0(\mathcal{E}_k)\to H^0(\mathcal{E}|_{\proj 2})$ are surjective for $k\gg 0$.
\end{proof}
\noindent We now state an easy generalization of a well known embedding criterion for algebraic manifolds.
Recall that a superscheme $\mathcal{Z}$ of super dimension $0|m$ is affine and  it has $\dim_{\mathbb{C}}(\mathcal{O}_{\mathcal{Z}})<\infty$. This dimension is called the {\em length} of $\mathcal{Z}$.
\begin{prop} A morphism of algebraic or complex supermanifolds $\Phi\colon \mathcal{M}\to\mathcal{N}$ is an embedding if and only if for any sub-superscheme $\mathcal{Z}\subset \mathcal{M}$ of length $2$ one has that the composition $\mathcal{Z}\to \mathcal{M}\to\mathcal{N}$ is a sub-superscheme, that is the induced map $\mathcal{O}_{\mathcal{N}}\to\mathcal{O}_{\mathcal{Z}}$ is surjective.
\end{prop} 
\begin{proof}[Sketch of proof] It is an immediate generalization of the analogous result in algebraic geometry (see for example \cite{Arrondo}, Proposition 2.4), taking into account that the possible supercommutative algebras $\mathcal{O}_{\mathcal{Z}}$ of dimension $2$ are of the form $\mathbb{C}_p\times \mathbb{C}_q$, with $p,q$ two distinct points and $\mathbb{C}_p$, $\mathbb{C}_q$ the skyscraper algebras equal to $\mathbb{C}$ over these points, or $\mathbb{C}_p[\varepsilon]$, with $\varepsilon$ of parity $0$ or $1$. The  composition $\mathcal{Z}\to \mathcal{M}\to\mathcal{N}$ allows one to take care of the injectivity of $\Phi$ on points in the first case and the injectivity of the tangent map $d\Phi$ on even or odd tangent vectors, in the other two cases.
\end{proof}
\noindent For the special case of embeddings into Grassmannians one can give more precise conditions.
\begin{prop}\label{prop:suffemb} Let $\Phi\colon \mathcal{M}\to G(a|b, V)$ be a map induced by some  epimorphism $V\otimes\mathcal{O}_{\mathcal{M}}\to\mathcal{E}$ of locally-free sheaves of $\mathcal{O}_\mani$-modules. Then $\Phi$ is an embedding if and only if for any sub-superscheme $\mathcal{Z}\subset\mathcal{M}$ of length $2$, the induced map $V\to\mathcal{E}\otimes_{\mathcal{O}_\mathcal{M}}\mathcal{O}_{\mathcal{Z}}$ has rank $>a|b$.  In all cases when $\mathcal{Z}$ has only one closed point $x\in \mathcal{M}$, that is when  $\mathcal{O}_{\mathcal{Z}}=\mathbb{C}_x[\varepsilon]$, a sufficient condition for the condition above to be satisfied is that $V\to \mathcal{E}/ \mathfrak{m}_x^2\mathcal{E}$ is surjective, with $\mathfrak{m}_x$ the maximal ideal of $x$ in $\mathcal{O}_{\mathcal{M},x}$.
\end{prop}
\begin{proof}[Sketch of proof] The necessary and sufficient condition says that the composition $\mathcal{Z}\to \mathcal{M}\to G(a|b, V)$ is not a {\em constant map}. To understand the last sufficient condition, one observes that, denoting $\mathcal{M}_x^{(2)}$ the affine sub-superscheme of $\mathcal{M}$ with support $\{x\}$ and associated algebra $\mathcal{O}_{\mathcal{M},x}/\mathfrak{m}_x^2$, one has a factorization $\mathcal{Z}\to \mathcal{M}_x^{(2)}\to \mathcal{M}\to G(a|b, V)$, because of the property $\varepsilon^2=0$. Then $V\to \mathcal{E}\otimes_{\mathcal{O}_\mathcal{M}}\mathcal{O}_{\mathcal{Z}}$ is the composition of $V\to \mathcal{E}/ \mathfrak{m}_x^2\mathcal{E}$ and the canonical  surjective morphism $\mathcal{E}/ \mathfrak{m}_x^2\mathcal{E}\to \mathcal{E}\otimes_{\mathcal{O}_\mathcal{M}}\mathcal{O}_{\mathcal{Z}}$.
\end{proof}
\noindent Now we can prove our main theorem \ref{embeddingth}. \vspace{8pt}\\
\noindent \emph{Proof of Theorem \ref{embeddingth}.}  By the preliminary results above, we have shown that for $k\gg 0$ the morphism $\Phi_k: \mani\to G(a|b, V)$ is globally defined, with $a|b$ the rank of the sheaf $Sym^k\mathcal{T}_\mani$. 

\noindent 
Note that $a|b$ can be computed from the formula (\ref{Skdecomp}) for the restriction of $Sym^k\mathcal{T}_\mani$ to $\proj 2$, where its even and odd summands are, respectively, 
$(Sym^k\mathcal{T}_\mani)_0=Sym^k\mathcal{T}_{\proj 2}\oplus  (Sym^{k-2} \mathcal{T}_{\proj 2}\otimes Sym^2 \mathcal{F}_\mani^\ast)$ and $(Sym^k\mathcal{T}_\mani)_1= Sym^{k-1} \mathcal{T}_{\proj 2}\otimes \mathcal{F}^\ast_\mani$, from which we get $a|b=2k|2k$. 

At the level of the reduced manifolds, $\Phi_k$ defines a morphism $\Phi_k |_{\proj 2}=(\phi_0,\phi_1) \colon \proj 2\longrightarrow G_0\times G_1$ which is associated to the surjections 
\bear
\xymatrix@R=1.5pt{
V_0\otimes\mathcal{O}_{\proj 2} \ar@{>>}[r] & H^0((\mathcal{E}_k)_0|_{\proj 2})\otimes\mathcal{O}_{\proj 2} \ar@{>>}[r] & (\mathcal{E}_k)_0|_{\proj 2} \nonumber \\
V_1\otimes\mathcal{O}_{\proj 2}\ar@{>>}[r] &  H^0((\mathcal{E}_k)_0|_{\proj 2})\otimes\mathcal{O}_{\proj 2}\ar@{>>}[r] &   (\mathcal{E}_k)_1|_{\proj 2}.\nonumber
}  
\eear
They define embeddings of $\proj 2$ into the ordinary Grassmannians $G_0 = G(2k; V_0)$ and $G_1 = G(2k; V_1)$ for $k\gg 0$ by well-known vanishing theorems in projective algebraic geometry.  This takes care of the injectivity of $\Phi_k$ at the point set level.\\
 By Proposition \ref{prop:suffemb}, the evaluation map $H^0(E)\to E$ defines an embedding into a Grassmannian if, when composed with the restriction $E\to E/ \mathfrak{m}_x^2E$, for any $x\in \proj 2$ and $\mathfrak{m}_x$ the maximal ideal in the stalk $\mathcal{O}_{\proj 2 , x}$, one gets a surjection 
$H^0(E)\to E/ \mathfrak{m}_x^2E=H^0(E/ \mathfrak{m}_x^2E)$. This is part of the exact sequence of cohomology associated to 
\bear
\xymatrix@R=1.5pt{
0\ar[r] & \mathfrak{m}_x^2 E \ar[r] & E \ar[r] & \slantone{E}{\mathfrak{m}_x^2E}\ar[r] & 0, \nonumber
}
\eear
so the surjection above is a consequence of $H^1(\mathfrak{m}_x^2E)=0$. In our case $E$ is either $E=\mathcal{E}_0\lfloor_{\proj 2}=Sym^k\mathcal{T}_{\proj 2}\oplus Sym^{k-2}\mathcal{T}_{\proj 2}(-3)$ or $E=\mathcal{E}_1|_{\proj 2}=Sym^{k-1}\mathcal{T}_{\proj 2}\otimes \mathcal{F}_\mani$, and the vanishing of $H^1(\mathfrak{m}_x^2E)=0$ can be shown in either case by means of the Euler sequence, by the same arguments as above.

\noindent
In conclusion, we have shown that $\Phi_k:\mani\to G(2k|2k, V)$ is injective at the level of geometrical points.

A similar criterion as in the ordinary algebraic geometry case applies to show the injectivity of the tangent map $d\Phi_k(x): \mathcal{T}_\mani(x)\to \mathcal{T}_{G(2k|2k,V)}(x)$ at any geometrical point $x\in \proj 2$. The maximal ideal of $x$ in 
$\mathcal{O}_{\mani, x}$ is $\mathfrak{M}_x \defeq \mathfrak{m}_x+\mathcal{J}_{\mani,x}$, and one can define the sub superscheme $\mathcal{V}_x$  of $\mani$ with reduced manifold $\{x\}$ and structure sheaf $\mathcal{O}_{\mani, x}/\mathfrak{M}_x^2$. Note that $(\mathfrak{M}_x^2)_0=\mathfrak{m}_x^2+\mathcal{J}_{\mani,x}^2$ 
and $(\mathfrak{M}_x)_1=\mathfrak{m}_x \mathcal{J}_{\mani,x}$, from which it follows $\mathcal{O}_{\mani, x}/\mathfrak{M}_x^2\cong \mathcal{O}_{\proj 2}/\mathfrak{m}_x^2\oplus (\mathcal{J}_{\mani,x}/\mathfrak{m}_x\mathcal{J}_{\mani,x})=\mathcal{O}_{\proj 2}/\mathfrak{m}_x^2\oplus \mathcal{F}(x)$. Note also that the tangent space of the 
superscheme $\mathcal{V}_x=(x,\mathcal{O}_{\mani, x}/\mathfrak{M}_x^2)$ is the same as the tangent space $\mathcal{T}_\mani(x)=(\mathfrak{M}_x/\mathfrak{M}_x^2)^\ast$. From these observations one gets the analogous result as in the classical case that the surjectivity of the 
restriction map $H^0(\mathcal{E}_k)\to H^0(\mathcal{E}_k\otimes \mathcal{O}_{\mani, x}/\mathfrak{M}_x^2 )=H^0(\mathcal{E}_k/\mathfrak{M}_x^2\mathcal{E}_k)$ ensures the injectivity of the tangent map $d\Phi_k$. Moreover observe that the superscheme embedding 
$\mathcal{V}_x\to \mani$ factorises through $\mani^{(2)}$, as $\mathcal{O}_{\mani, x}/\mathfrak{M}_x^2$ is also a $\mathcal{O}_{\mani}/\mathcal{J}_\mani^2$-module. Then the restriction map factorises as follows 
\bear
\xymatrix@R=1.5pt{
H^0(\mathcal{E}_k) \ar@{>>}[r]  & H^0(\mathcal{E}_k|_{\mani^{(2)}}) \ar[r]  &  H^0(\mathcal{E}_k/\mathfrak{M}_x^2\mathcal{E}_k), \nonumber
}
\eear
and we will show that the second map is surjective as well, using the fact that $\mathcal{E}_k\lfloor_{\mani^{(2)}}$ is a $\mathcal{O}_{\proj 2}$-module and by applying similar arguments as above, based on the vanishing of the higher cohomology of $H^i(\proj 2, \mathcal{G}(k))$, with $\mathcal{G}$ any fixed coherent sheaf, for $k\gg 0$. Indeed in our case we have $\mathcal{E}_k|_{\mani^{(2)}}\cong \mathcal{E}_k|_{\proj 2}\oplus (\mathcal{E}_k|_{\proj 2}\otimes \mathcal{F}_\mani)$ as a $\mathcal{O}_{\proj 2}$- module, so the decomposition (\ref{Skdecomp}) still applies to give the structure of $\mathcal{E}_k|_{\mani^{(2)}}$ as a $\mathcal{O}_{\proj 2}$- module. Setting $\overline{\mathfrak{M}_x^2}$ the ideal sheaf of $V_x$ in $\mani^{(2)}$, one has the exact sequence $0\to \overline{\mathfrak{M}_x^2}\mathcal{E}_k|_{\mani^{(2)}}\to \mathcal{E}_k|_{\mani^{(2)}}\to \mathcal{E}_k/\mathfrak{M}_x^2\mathcal{E}_k\to 0$, so we are left to prove $H^1(\overline{\mathfrak{M}_x^2}\mathcal{E}_k|_{\mani^{(2)}})=0$ for $k\gg 0$.  Now $\overline{\mathfrak{M}_x^2}\mathcal{E}_k|_{\mani^{(2)}}=
\overline{\mathfrak{M}_x^2}\mathcal{E}_k|_{\proj 2}\oplus (\overline{\mathfrak{M}_x^2}\mathcal{E}_k|_{\proj 2}\otimes \mathcal{F}_\mani)$ as a $\mathcal{O}_{\proj 2}$- module, therefore the decomposition (\ref{Skdecomp}) and the Euler sequences (\ref{SymmEuler}) apply to our case, showing that $H^1(\overline{\mathfrak{M}_x^2}\mathcal{E}_k|_{\mani^{(2)}})=0$ holds because of the vanishing of the higher cohomology groups $H^i(\proj 2, \mathcal{G}(k))$ for $k\gg 0$, with 
$\mathcal{G}=\overline{\mathfrak{M}_x^2}\otimes \mathcal{H}$, where $\mathcal{H}$ is any of the sheaves $\mathcal{O}_{\proj 2},\  \mathcal{F}^\ast_\mani,\  Sym^2\mathcal{F}^\ast_\mani,\  \mathcal{F}_\mani,\  \mathcal{F}_\mani\otimes\mathcal{F}^\ast_\mani,\  \mathcal{F}_\mani\otimes Sym^2\mathcal{F}^\ast_\mani$.
\begin{remark} Theorem \ref{embeddingth} is not effective, since it does not give any estimate on $k$ and on the super dimension of $V=H^0(\mathcal{E}_k)$ and hence it does not identify the target super Grassmannian of the embedding $\Phi_k$. In fact 
$k$ depends heavily on the choice of $\mathcal{F}_\mani$. However, it seems possible to calculate a uniform $k$  and $\dim V$ under some boundedness conditions on $\mathcal{F}_\mani$, such as $\mathcal{F}_\mani^\ast$ globally-generated, or $\mathcal{F}_\mani$ semistable.
\end{remark}
\begin{remark}
If one wants to generalize the result of Theorem \ref{embeddingth} to other non-projected supermanifolds, with reduced manifold $\manir$ with $\dim \manir\geq 2$, then the tangent sheaf $\mathcal{T}_{\manir}$ will not in general be ample (this happens only for $\manir$ a 
projective space, since projective spaces $\mathbb{P}_{\mathbb{C}}^n$ are the only projective varieties with ample tangent bundle, by a celebrated theorem of S. Mori, see \cite{Mori}) and therefore one faces the problem of finding a suitable ample locally-free sheaf of $\stsheafred$-modules $E$ on $\manir$ that {\em can be extended} to a locally-free sheaf $\mathcal{E}$ on $\mani$. This is a delicate problem that we will address in a future work.
\end{remark}
\noindent Before we go on to the next section we propose the following
\begin{problem} Find a \emph{fixed} super Grassmannian $G=G(2k|2k, V)$, i.e. a uniform $k$ and $\dim V$, so that $\mani=\proj2_{\omega} (\mathcal{F}_\mani)$ can be embedded in $G$, in the case when $\mathcal{F}^\ast_\mani$ is ample, or in the case when it is stable, with given 
$c_1(\mathcal{F}_\mani)=-3$ and $c_2(\mathcal{F}_\mani)=n$.
\end{problem}

%%%%%%%%%%%%%%%%%%%%%%%%%%%%%%%%%%%%%%%%%%%%%%%%%%%%%%%%%%%%%%%%%%%%%
\section{$\Pi$-projectivity versus non-$\Pi$-projectivity}

\noindent In Theorem \ref{thm:nonproj} we have seen that all the non-projected supermanifolds over $\proj 2$ are non-projective, in that they do not possess any even invertible sheaf. On the other hand in Theorem \ref{embeddingth} we have shown that they can all be embedded in some super Grassmannian, by the use of suitable locally-free sheaves of $\stsheaf$-modules on these supermanifolds. 

 In \cite{Manin}, Manin suggested that the notion of invertible sheaves might no longer be fundamental in supergeometry and he proposed instead the notion of $\Pi$-invertible sheaves to be the right one to employ, together with a related 
notion of $\Pi$-projective spaces as suitable embedding spaces, which would be nice since $\Pi$-projective spaces are simpler supermanifolds endowed with a similar universal property as projective spaces, than more general super Grassmannians. We defer to subsection \ref{subsection:Pi} for the precise definitions of $\Pi$-projective spaces and $\Pi$-invertible sheaves.
But, assuming those notions are already set up, the  following natural question arises.
\begin{question} Can any non-projected supermanifold of the form $\proj {2}_\omega (\mathcal{F}_\mani)$  be embedded in some $\Pi$-projective space?
\end{question}
\noindent In the present section we will show that in general the answer to the question above is {\em negative}, and indeed the possibility of embedding non-projected supermanifolds of the form $\proj {2}_\omega (\mathcal{F}_\mani)$ into $\Pi$-projective spaces very strongly depends on the choice of $\mathcal{F}_{\mani}$. We will do so by considering the following two extreme cases with $\mathcal{F}_\mani$ a  homogeneous vector bundle on $\mathbb{P}^2$. 
\begin{itemize}
\item {\bf decomposable}:  $\mathcal{F}_\mani \defeq \Pi \mathcal{O}_{\proj 2} (-1) \oplus \Pi \mathcal{O}_{\proj 2} (-2).$ 
\item {\bf non-decomposable}: $\mathcal{F}_\mani \defeq \Pi \Omega^1_{\proj 2}$.
\end{itemize}

\noindent We will prove  that the  first supermanifold above is \emph{not} $\Pi$-projective, that is it cannot be 
embedded in a $\Pi$-projective space. Therefore, under these circumstances, the notion of $\Pi$-invertible sheaves does not prove useful to get further geometrical knowledge of the supermanifold. 
\vskip2mm
On the other hand, referring to previous work of one of the authors, we will recall that the non-projected supermanifold with $\mathcal{F}_\mani \defeq \Pi \Omega^1_{\proj 2}$ is the $\Pi$-projective plane itself, which of course answers the question of existence of an embedding in a $\Pi$-projective space in the affirmative trivial way, in this case. 
\vskip2mm
First of all we give a more detailed description of the two supermanifolds introduced above. Although not strictly necessary for the remainder of this paper, we will provide for each of the two non-projected supermanifolds an explicit atlas and transition functions, hoping in this way to give the reader a more concrete perception of the objects at hand.

\subsection{Decomposable Sheaf: $\mathcal{F}_\mani = \Pi \mathcal{O}_{\proj 2} (-1) \oplus \Pi \mathcal{O}_{\proj 2} (-2)$}  We have the following result.
\begin{prop}[Transition functions (1)] Let $\proj 2_\omega(\mathcal{F}_\mani)$ be the non-projected supermanifold with
$\mathcal{F}_\mani = \Pi \mathcal{O}_{\proj 2} (-1) \oplus \Pi \mathcal{O}_{\proj 2} (-2)$. Then, its transition functions take the following form:
\begin{align}\label{trans1}
& \mathcal{U}_{0} \cap \mathcal{U}_{1} : \qquad z_{10} = \frac{1}{z_{11}}, \qquad \quad z_{20} = \frac{z_{21}}{z_{11}} + \lambda \frac{\theta_{11} \theta_{21}}{(z_{11})^2};  & \qquad \theta_{10} = \frac{\theta_{11}}{z_{11}}, \qquad \theta_{20}  = \frac{\theta_{21}}
{(z_{11})^2};\nonumber \\  
& \mathcal{U}_{1} \cap \mathcal{U}_{2} : \qquad  z_{11} = \frac{z_{12}}{z_{22}} + \lambda \frac{\theta_{12} \theta_{22}}{(z_{22})^2}, \quad \qquad  z_{21} = \frac{1}{z_{22}}; & \qquad  \theta_{11} = \frac{\theta_{12}}{z_{22}}, \qquad \theta_{21}  = \frac{\theta_{22}}
{(z_{22})^2}; \nonumber\\  
& \mathcal{U}_{2} \cap \mathcal{U}_{0} : \qquad  z_{12} = \frac{1}{z_{20}}, \qquad  \quad z_{22} = \frac{z_{10}}{z_{20}} + \lambda \frac{\theta_{10} \theta_{20}}{(z_{20})^2}; & \qquad  \theta_{12} = \frac{\theta_{10}}{z_{10}},  \qquad \theta_{22} = 
\frac{\theta_{20}}{(z_{10})^2}.  
\end{align} 
\end{prop} 

\begin{proof} It follows immediately from Theorem \ref{gentransition}, taking into account the transition matrices for $\mathcal{F}_\mani$, that have the form $M=\left(\begin{array}{ll} \frac{1}{z_{01}} &0\\ 0& \frac{1}{z_{01}^2}\end{array}\right)$ on 
$\mathcal{U}_{0} \cap \mathcal{U}_{1}$ and similar forms on the other two intersections of the fundamental open sets.  \end{proof}

\subsection{Non-Decomposable Sheaf: $\mathcal{F}_\mani = \Pi \Omega^1_{\proj 2} $}

If we take $\Pi \Omega^1_{\proj 2} $ to be the fermionic sheaf of the supermanifold $\proj 2_\omega$, then we let $\theta_{10}$, $\theta_{20}$ transform as $dz_{10}$ and $dz_{20}$, respectively, obtaining the transformations 
\begin{align} \label{transOmega}
\mathcal{U}_{0} \cap \mathcal{U}_{1} : \qquad \quad & \theta_{10} = - \frac{\theta_{11}}{(z_{11})^2}, \qquad \quad \theta_{20} = - \frac{z_{21}}{(z_{11})^2} \theta_{11} + \frac{\theta_{21}}{z_{11}}; \nonumber \\
\mathcal{U}_{2} \cap \mathcal{U}_{0} : \qquad \quad & \theta_{12} = - \frac{\theta_{20}}{(z_{20})^2}, \qquad \quad  \theta_{22} =  \frac{\theta_{10}}{z_{10}} - \frac{z_{10}}{(z_{20})^2}\theta_{20}; \nonumber \\ 
\mathcal{U}_{1} \cap \mathcal{U}_{2} : \qquad \quad  & \theta_{11} = - \frac{z_{12}}{(z_{22})^2} \theta_{22} + \frac{\theta_{12}}{z_{22}}, \qquad \quad \theta_{21} = - \frac{\theta_{22}}{(z_{22})^2}.  
\end{align}
Just like above, we now look for the complete form of the transition functions. By Theorem \ref{gentransition}, we have the following result.
\begin{prop}[Transition functions (2)] Let $\proj 2_\omega$ be the non-projected supermanifold with $\mathcal{F}_\mani = \Pi \Omega^1_{\proj 2}$. Then, its transition functions take the following form:
\begin{align} \label{trans2}
& \mathcal{U}_{0} \cap \mathcal{U}_{1} : \qquad z_{10} = \frac{1}{z_{11}}, \quad z_{20} = \frac{z_{21}}{z_{11}} + \lambda \frac{\theta_{11} \theta_{21}}{(z_{11})^2}; & \quad \theta_{10} = - \frac{\theta_{11}}{(z_{11})^2},  \quad \theta_{20} = - \frac{z_{21}}
{(z_{11})^2} \theta_{11} + \frac{\theta_{21}}{z_{11}}; \nonumber \\  
& \mathcal{U}_{1} \cap \mathcal{U}_{2} : \qquad  z_{11} = \frac{z_{12}}{z_{22}} - \lambda \frac{\theta_{12} \theta_{22}}{(z_{22})^2}, \quad  z_{21} = \frac{1}{z_{22}} ;  & \quad  \theta_{12} = - \frac{\theta_{20}}{(z_{20})^2}, \quad  \theta_{22} =  \frac{\theta_{10}}
{z_{10}} - \frac{z_{10}}{(z_{20})^2}\theta_{20};\nonumber  \\  
& \mathcal{U}_{2} \cap \mathcal{U}_{0} : \qquad  z_{12} = \frac{1}{z_{20}}, \quad z_{22} = \frac{z_{10}}{z_{20}} - \lambda \frac{\theta_{10} \theta_{20}}{(z_{20})^2}; & \quad \theta_{11} = - \frac{z_{12}}{(z_{22})^2} \theta_{22} + \frac{\theta_{12}}{z_{22}}, \quad 
\theta_{21} = - \frac{\theta_{22}}{(z_{22})^2}. 
\end{align} 
\end{prop}
\begin{proof} Again, it follows immediately from Theorem \ref{gentransition}, taking into account the transition functions for $\mathcal{F}_\mani$ provided by (\ref{transOmega}).  \end{proof}

\subsection{$\Pi$-projective spaces and $\Pi$-invertible sheaves}\label{subsection:Pi}
As $\Pi$-projective geometry is not a central topic of this paper, we refer to the literature for an introduction to the subject,  in particular \cite{Manin} Chapter 5, \S 6.4  and \cite{ManinNC} Chapter 2, \S 8.5 and \S 8.11. In what follows we recall the basic definitions of $\Pi$-projective spaces and $\Pi$-invertible sheaves.
We start with the following notions.

\begin{definition}[$\Pi$-symmetric modules] Let $M$ be a supercommutative free $A$-module such that $M=A^n \oplus \Pi A^n $. Let $p_\Pi$ denote the odd involution $p_\Pi:M\to \Pi M$ that exchanges the even and odd base elements. Then we say that a super submodule $S\subset M$ is $\Pi$-symmetric if it is stable under the action of $p_\Pi.$   In particular $M$ itself is $\Pi$-symmetric.
\end{definition}
\begin{definition}[$\Pi$-symmetric locally-free sheaves] Let $\mathcal{G}$ be a locally-free sheaf of $\stsheaf$-modules of rank $n|n$ on a supermanifold $\mani$. We say that $\mathcal{G}$ is a $\Pi$-symmetric locally-free sheaf if it comes together with a $\Pi$-symmetry, that is an {\em odd} involution $p_{\Pi} : \mathcal{G} \rightarrow \Pi \mathcal{G}$ such 
that $p_{\Pi }^2 = id$. 
\end{definition}
\noindent In particular on every open set $\mathcal{U} \subset |\mani |$ such that $M=\mathcal{G}(\mathcal{U})$ is free, it is a $\Pi$-symmetric $A=\mathcal{O}_\mani(\mathcal{U})$-module as in the definition above. Moreover $H^0(\mani,\mathcal{G})$ is a $\Pi$-symmetric $\mathbb{C}$-superspace and one can define the subspace of $p_\Pi$-invariant global sections $H^0_\Pi(\mani,\mathcal{G})\subset H^0(\mani,\mathcal{G})$, which is of course $\Pi$-symmetric. 

\begin{definition}[$\Pi$-invertible sheaves] A $\Pi$-invertible sheaf 
$\mathcal{G}_\Pi$ is a $\Pi$-symmetric locally-free sheaf of rank $1|1$. In other words, it is a locally-free sheaf of $\mathcal{O}_\mani$-modules of rank $1|1$ together with an odd involution $p_{\Pi} : \mathcal{G}_\Pi \rightarrow \Pi \mathcal{G}_{\Pi}$.
\end{definition} 

\noindent Locally, on an open set $\mathcal{U} \subset |\mani |$, the involution $p_\Pi$
exchanges the even and odd components of the sheaf $\mathcal{G}_\Pi |_{\mathcal{U}} \cong \mathcal{O}_{\mani} (\mathcal{U}) \oplus \Pi \mathcal{O}_{\mani} ({\mathcal{U}})$. 
\begin{definition}[See \cite{ManinNC} Chapter 2, \S 8.5] Let $T=\mathbb{C}^{n+1|n+1}$ be the super-vector space endowed with the odd involution $p\colon T\to T$, with $p^2=\operatorname{id}$. Then the $\Pi$-projective space $\mathbb{P}_\Pi^{n}$ is the Grassmannian $G\Pi(1|1, T)$ of the $p$-invariant $1|1$-subspaces of $T$. It has a tautological $1|1$ locally free sheaf $\mathcal{O}_\Pi(1)$ that is also endowed with a odd involution $p$.
\end{definition}
In \cite{ManinNC} Chapter 2, \S 8.5 it is also observed that to give a morphism $f: \mani \to \mathbb{P}_\Pi^{n}$ is equivalent to give a $\Pi$-invertible sheaf $\mathcal{G}_\Pi$ on $\mani$ and a sheaf epimorphism  $T^\ast\otimes\mani\to \mathcal{G}_\Pi$, compatible with the two involutions $p$ on $T$ and $p_\Pi$ on $\mathcal{G}_\Pi$ and such that $\mathcal{G}_\Pi=f^\ast\mathcal{O}_\Pi(1)$. We observe that in this case the space of $p$-invariant elements of $T^\ast$ produce $p_\Pi$- invariant sections of $\mathcal{G}_\Pi$ that do not have common zeros on $\mani$, and conversely, as Manin observes in \cite{ManinNC} Chapter 2, \S 8.5, any choice of $n+1$ such sections gives rise to a morphism to some $\mathbb{P}^n_\Pi$. In particular one can easily see that if $H^0_\Pi(\mathcal{G}_\Pi)$ is a $1$-dimensional vector space, then the map $f$ is constant. 

As noted by Manin in \cite{ManinNC} Chapter 2, \S 8.11, giving an odd involution on a rank $1|1$ sheaf corresponds to reduce its structure group, the whole super Lie group $GL(1|1, \mathcal{O}_\mani)$, to the non-commutative multiplicative group 
$\mathbb{G}^{1|1}_m (\mathcal{O}_\mani)$. Likewise, the set of isomorphism classes of $\Pi$-invertible sheaves on a certain supermanifold $\mani$, denoted with $\mbox{Pic}_\Pi (\mani)$ by similarity with the usual Picard group, can be identified with 
the \emph{pointed set} $H^1(\mathbb{G}^{1|1} (\mathcal{O}_\mani))$. \\
The embedding $\mathbb{G}_m \hookrightarrow \mathbb{G}^{1|1}_m$, induces a map as follows
\begin{align}
i : \mbox{Pic}_0 (\mani)  \longrightarrow \mbox{Pic}_\Pi (\mani), \qquad\qquad\ \mathcal{L}_\mani  \longmapsto \mathcal{L}_\mani\oplus \Pi \mathcal{L}_\mani,
\end{align} 
where $\mathcal{L}_\mani$ is a locally-free sheaf of $\mathcal{O}_\mani$-modules of rank $1|0$ (generalisation of usual line bundles, as above) and the $\Pi$-invertible sheaf $\mathcal{L}_\mani\oplus \Pi \mathcal{L}_\mani$ is called the \emph{interchange 
of summands}, to stress that it comes endowed with the morphism $p_\Pi$. We say that a $\Pi$-invertible sheaf \emph{splits} if it is isomorphic to the interchange of summands $\mathcal{L}_\mani \oplus \Pi \mathcal{L}_\mani$. Analogously, we might have said that 
a $\Pi$-invertible sheaf splits if its structure group $\mathbb{G}^{1|1}_m$ can in turn be reduced to the usual $\mathbb{G}_m.$ \\
The injective map $\mathbb{G}_m \rightarrow \mathbb{G}^{1|1}_m$ fits into an exact sequence as follows (see again \cite{ManinNC})
\bear
\xymatrix@R=1.5pt{
1 \ar[r] & \mathbb{G}_m \ar[r] & \mathbb{G}_{m}^{1|1}  \ar[r] & \mathbb{G}^{0|1}_a \ar[r] &  0, 
}  
\eear
that is useful to study whenever a $\Pi$-invertible sheaf splits. Indeed, as $\mathbb{G}_m$ is central in $\mathbb{G}^{1|1}_m$, the sequence of pointed sets corresponding to the first \v{C}ech cohomology groups associated to short exact sequence above 
can be further extended to $H^2(\mathbb{G}_m(\mathcal{O}_\mani)) = H^2 (\mathcal{O}_{\mani, 0}^\ast)$, giving 
\bear
\xymatrix@R=1.5pt{
\cdots \ar[r] & \mbox{Pic}_0 (\mani) \ar[r] & \mbox{Pic}_{\Pi} (\mani) \ar[r] & H^1 (\mathcal{O}_{\mani, 1}) \ar[r]^\delta &  H^2 (\mathcal{O}_{\mani, 0}^\ast).
}  
\eear 
Clearly, the obstruction to splitting of $\Pi$-invertible sheaves for a supermanifolds lies in the image of the map $\mbox{Pic}_{\Pi} (\mani) \rightarrow H^1 (\mathcal{O}_{\mani, 1}) $ or, analogously, by exactness, in the kernel of 
$H^1 (\mathcal{O}_{\mani, 1}) \rightarrow H^2 (\mathcal{O}_{\mani, 0}^\ast)$. \\
We apply the considerations above to obtain the following result.
\begin{theorem}\label{thm:nonPiproj}
Let $\mani=\proj 2_\omega (\mathcal{F}_\mani)$ with fermionic sheaf $\mathcal{F}_\mani = \Pi \mathcal{O}_{\proj 2} (-1) \oplus \Pi \mathcal{O}_{\proj 2} (-2)$. Then $\mbox{\emph{Pic}}_{\Pi} (\mani)$ is just a point, representing the trivial $\Pi$-invertible sheaf 
$\mathcal{O}_\mani\oplus\Pi\mathcal{O}_\mani$. In particular $\mani$ cannot be embedded in a $\Pi$-projective space.
\end{theorem}
\begin{proof}
Remembering that $\mathcal{F}_\mani \cong \mathcal{O}_{\mani, 1}$, as the supermanifold has dimension $2|2$, one easily compute that 
\bear
H^1  (\mathcal{O}_{\mani, 1}  ) \cong H^1 (\mathcal{F}_\mani) \cong H^1 (\mathcal{O}_{\proj 2} (-1)) \oplus H^1 (\mathcal{O}_{\proj 2} (-2)) = 0.
\eear
This tells us that we have a surjection
\bear
\mbox{Pic}_{0} (\mani) \longrightarrow \mbox{Pic}_{\Pi} (\mani) \longrightarrow 0. 
\eear
and therefore all the $\Pi$-invertible sheaves will be of the form $\mathcal{L}_\mani \oplus \Pi \mathcal{L}_\mani$. On the other hand we do already know that the even Picard group of $\mani$ is trivial, and the only invertible sheaf of rank $1|0$ is actually the structure sheaf. This tells us that the only $\Pi$-invertible sheaf that can be defined on $\mani$ endowed with a decomposable fermionic sheaf as above is given by 
$\mathcal{G}_\Pi \defeq \mathcal{O}_{\mani} \oplus \Pi \mathcal{O}_{\mani}.$ We have 
\bear
\mbox{Pic}_\Pi (\mani ) = \{\mathcal{O}_{\mani} \oplus \Pi \mathcal{O}_{\mani} \},
\eear
that is the pointed-set $\mbox{Pic}_{\Pi } (\mani)$ is given by its base point only. 
Clearly, as there are no non-trivial $\Pi$-invertible sheaves there is no hope for $\mani$ to be embedded in a $\Pi$-projective space. 
\end{proof}
The scenario is much different when one considers $\proj 2_\omega (\mathcal{F}_\mani)$ endowed with the non-decomposable fermionic sheaf $\mathcal{F}_\mani = \Pi \Omega^1_{\proj 2}$.

The $\Pi$-projective plane $\proj {2}_\Pi$, that describes $1|1$-dimensional $\Pi$-symmetric subspaces of $\mathbb{C}^{3|3}$, is covered by three affine charts, whose coordinates in the super big-cell notation, as shown in \cite{Manin} Chapter 5, \S 6.4, are given by
\bear
\mathcal{Z}_{\mathcal{U}_0} = \left ( \begin{array}{ccc||ccc}
1 & z_{10} & z_{20} & 0 \; & \theta_{10} & \theta_{20} \\
\hline \hline
0 & - \theta_{10} & - \theta_{20} & 1 \; & z_{10} & z_{20} \\
\end{array}
\right ) \qquad 
\mathcal{Z}_{\mathcal{U}_1} = \left ( \begin{array}{ccc||ccc}
z_{11} & 1 & z_{21} & \theta_{11} & 0  \;& \theta_{21} \\
\hline \hline
- \theta_{11} & 0 & - \theta_{21} &  z_{11} & 1\; & z_{21} \\
\end{array}
\right ) \nonumber 
\eear
$$ \mathcal{Z}_{\mathcal{U}_2} = \left ( \begin{array}{ccc||ccc}
z_{12} & z_{22} & 1  & \theta_{12} & \theta_{22}  & \;0 \\
\hline \hline
- \theta_{12} & - \theta_{22} & 0 &  z_{12} & z_{22} & \; 1\\
\end{array}
\right )
$$
and setting the nilpotent coordinates to zero it is apparent that the underlying manifold is given by $\proj 2$.
\begin{theorem}[$\mani$ is $\proj{2}_\Pi$] The non-projected supermanifold given by $\mani \defeq (\proj 2, \Pi \Omega^1_{\proj 2}, \lambda)$ where $\lambda $ is a non-zero representative of $\omega \in H^1 (\mathcal{T}_{\proj 2} \otimes Sym^2 \Pi \Omega^1_{\proj 2} ) \cong \mathbb{C}$ is the $\Pi$-projective plane $\proj {2}_\Pi$.
\end{theorem}  
\begin{proof} We have already seen that the topological space underlying $\proj {2}_\Pi$ is $\proj {2}$. To prove that the two spaces are the same supermanifold we consider the structure sheaf $\mathcal{O}_{\proj {2}_\Pi}$ of $\proj {2}_\Pi $ and we prove 
that the transition functions among its affine charts coincide with those of $\mani.$ To this end, by allowed row operations we get
\begin{align}
& z_{10} = \frac{1}{z_{11}}, \qquad z_{20} = \frac{z_{21}}{z_{11}} + \frac{\theta_{11} \theta_{21}}{(z_{11})^2},  \qquad \theta_{10} = - \frac{\theta_{11}}{(z_{11})^2}, \qquad \theta_{20} = - \frac{z_{21}\theta_{11}}{(z_{11})^2} + \frac{\theta_{21}}{z_{11}}, \nonumber
\end{align}
these coincide with the transition functions we found in (\ref{trans2}), once it is set $\lambda = 1$ - which is always possible by means of a scaling or a change of coordinates. By analogous calculations one checks that the same happens in the other intersections, thus showing $\proj {2}_\Pi = (\proj 2, \Pi \Omega_{\proj n}, \lambda)$.
\end{proof}
\begin{remark} Indeed the result above is a particular case of a much more general result that relates $\Pi$-projective spaces and the cotangent sheaf $\Omega^1_{\proj n }$ of ordinary projective spaces, proved in
\cite{Noja},  Theorem 4.3.\end{remark}

\noindent The previous theorem has the following obvious corollary, which can be seen as an improvement and a quantitative version of the result in Theorem \ref{embeddingth} in the case at hand.  
\begin{corollary} The supermanifold $\mani \defeq (\proj 2, \Pi \Omega^1_{\proj 2}, \lambda)$ can be embedded into $G(1|1, \mathbb{C}^{3|3}).$ 
\end{corollary}
\begin{proof} Since we have shown that $\mani = \proj {2}_\Pi$ and the $\Pi$-projective plane $\proj {2}_\Pi$ can be presented as a closed sub-supermanifold inside $G(1|1, \mathbb{C}^{3|3})$, the same holds true for $\mani$ and we have a linear 
embedding of $\mani$ into $G(1|1, \mathbb{C}^{3|3}).$
\end{proof}

%%%%%%%%%%%%%%%%%%%%%%%%%%%%%%%%%%%%%%%%%%%%%%%%%%%%%%%%%%%%%%%%%%%%%%%%%%%%

%%%%%%%%%%%%%%%%%%%%%%%%%%%%%%%%%%%%%%%%%%%%%%%%%%%%%%%%
\noindent {\bf Acknowledgements:} We thank the anonymous referees for their remarks and suggestions that definitely helped us improving the quality and readability of the paper. We thank Bert van Geemen for several discussions and suggestions. We also thank Yuri Manin for his stimulating suggestions and his availability, and for having gotten us in touch with his former students Ivan Penkov, Igor Skornyakov, and Vera Serganova. In particular, we thank Ivan Penkov and Igor Skornyakov for very useful correspondence and for having sent us their previous work on super Grassmannians. S.N. thanks Ron Donagi for having suggested this stimulating research topic.

%%%%%%%%%%%%%%%%%%%%%%%%%%%%%%%%%%%%

\end{document}